\documentclass[11pt]{amsart}
\usepackage[T1]{fontenc} 
\usepackage{amsmath,amsthm,amsfonts,amssymb,bm,graphicx,xcolor,mathtools}

\usepackage%[backref=page]
{hyperref}
\hypersetup{colorlinks=true,citecolor=blue,linkcolor=blue,urlcolor=blue,
pdfstartview=FitH }

\theoremstyle{plain}
\newtheorem{theorem}{Theorem}

\newtheorem{lemma}{Lemma}

\mathtoolsset{showonlyrefs=true}
\theoremstyle{definition}

\newtheorem{remark}{Remark}

\renewcommand{\geq}{\geqslant}
\renewcommand{\leq}{\leqslant}

\newcommand{\p}{\mathbf{p}}

\newcommand{\ppmod}[1]{\,\hspace{-0.3cm}\pmod{#1}}

\newcommand{\sgn}{\operatorname{sgn}}

\newcommand{\bdd}{\begin{center}\begin{tikzcd}}
\newcommand{\bd}{\begin{tikzcd}}
\newcommand{\edd}{\end{tikzcd}\end{center}}
\newcommand{\ed}{\end{tikzcd}}
\newcommand{\bdp}{\begin{center}\begin{tikzpicture}}
\newcommand{\edp}{\end{tikzpicture}\end{center}}
\newcommand{\bi}{\begin{itemize}}
\newcommand{\ei}{\end{itemize}}
\newcommand{\bt}{\begin{tikzpicture}}
\newcommand{\et}{\end{tikzpicture}}
\newcommand{\ba}{\[\begin{aligned}}
\newcommand{\ea}{\end{aligned}\]}
\newcommand{\bp}{\begin{pmatrix}}
\newcommand{\ep}{\end{pmatrix}}
\newcommand{\bv}{\begin{vmatrix}}
\newcommand{\ev}{\end{vmatrix}}
\newcommand{\bb}{\begin{bmatrix}}
\newcommand{\eb}{\end{bmatrix}}
\newcommand{\bB}{\begin{Bmatrix}}
\newcommand{\eB}{\end{Bmatrix}}
\newcommand{\bea}{\begin{enumerate}[label=(\alph*)]}
\newcommand{\ber}{\begin{enumerate}[label=(\roman*)]}
\newcommand{\ben}{\begin{enumerate}[label=(\arabic*)]}
\newcommand{\ee}{\end{enumerate}}

\newcommand{\RR}{\mathbb{R}}

\numberwithin{equation}{section}

\newcommand{\cm}{\tilde{c}}
\newcommand{\cn}{\breve{c}}
\newcommand{\qm}{\tilde{q}}
\newcommand{\qn}{\breve{q}}

\makeatletter
\def\Ddots{\mathinner{\mkern1mu\raise\p@
\vbox{\kern7\p@\hbox{.}}\mkern2mu
\raise4\p@\hbox{.}\mkern2mu\raise7\p@\hbox{.}\mkern1mu}}
\makeatother
\makeatletter
\DeclareRobustCommand\widecheck[1]{{\mathpalette\@widecheck{#1}}}
\def\@widecheck#1#2{%
    \setbox\z@\hbox{\m@th$#1#2$}%
    \setbox\tw@\hbox{\m@th$#1%
       \widehat{%
          \vrule\@width\z@\@height\ht\z@
          \vrule\@height\z@\@width\wd\z@}$}%
    \dp\tw@-\ht\z@
    \@tempdima\ht\z@ \advance\@tempdima2\ht\tw@ \divide\@tempdima\thr@@
    \setbox\tw@\hbox{%
       \raise\@tempdima\hbox{\scalebox{1}[-1]{\lower\@tempdima\box
\tw@}}}%
    {\ooalign{\box\tw@ \cr \box\z@}}}
\makeatother

\begin{document}

\author{Valentin Blomer}
\address{Mathematisches Institut, Endenicher Allee 60, 53115 Bonn, Germany}
\email{blomer@math.uni-bonn.de}
 
\author{Junxian Li}
\address{University of California Davis,
Mathematics Department, 
One Shields Avenue, 
Davis, CA 95616, 
USA}
\email{junxian@math.ucdavis.edu}

 \title{A higher rank shifted convolution problem with applications to $L$-functions}

\thanks{First author is supported by DFG through SFB-TRR 358 and EXC-2047/1 - 390685813 and by ERC
Advanced Grant 101054336. The second author was supported by the Max Planck Institute for Mathematics and the NSF (DMS-2502537).}


\begin{abstract}  While several instances of shifted convolution problems for  ${\rm GL}(3) \times {\rm GL}(2)$ have been solved, the case where one factor is the classical divisor function and one factor is a ${\rm GL}(3)$ Fourier coefficient has remained open. We solve this case in the present paper. The proof involves two intertwined applications of different types of delta symbol methods. As an application we establish an asymptotic formula for central values of $L$-functions for a ${\rm GL}(3)$ automorphic form twisted by Dirichlet characters to moduli $q \leq Q$. 
\end{abstract}


\setcounter{tocdepth}{2}  \maketitle 

\maketitle

\section{Introduction}
 
\subsection{Shifted convolution problems}  A shifted convolution problem asks for an asymptotic formula for the product of two (usually) multiplicative arithmetic function whose arguments differ by an additive shift. It is therefore  a measure of the correlation of the two functions. The most classical case for the divisor function $\tau  = \textbf{1} \ast \textbf{1}$ is
 $$\sum_{n \leq x} \tau(n)\tau(n+1),$$
 which has been investiagted from various points of view for a century. Definitive results exist also in the case when the divisor function is replaced with Fourier coefficients of ${\rm GL}(2)$ automorphic forms. In this case, the arithmetic function cannot be opened by a convolution formula, but   a delta symbol can be used instead which has roughly the same strength. We recall that the divisor function can be seen as a Fourier coefficients of an Eisenstein series, and both types of arithmetic functions share a structurally similar Voronoi summation formula. 

\medskip

Shifted convolution problems with Fourier coefficients of higher rank automorphic forms, cuspidal or non-cuspidal, turn out to be extremely challanging and non-trivial results are unknown in most cases. Progress has been made in the case when one factor is associated with an automorphic form on ${\rm GL}(3)$ and the other  is associated with an automorphic form on ${\rm GL}(2)$. The most factorable case
\begin{equation}\label{tau3tau}
\sum_{n \leq x} \tau_3(n)\tau(n+1),
\end{equation}
where $\tau_3 = \textbf{1} \ast \textbf{1} \ast \textbf{1}$ denotes the ternary divisor  function, was first treated by Hooley \cite{Ho} who obtained the main term in the asymptotic formula. The first power saving error term was established by  Deshouillers \cite{De} based on the Kuznetsov formula, and the current record $O(x^{5/6 + \theta/3 +\varepsilon})$, $\theta = 7/64$ being an admissible constant towards the Ramanujan conjecture, for the error term of a smooth version of \eqref{tau3tau} is due to Topacogullari \cite{To}. 

In the case when the divisor function $\tau$ in \eqref{tau3tau} is replaced with a ${\rm GL}(2)$ Fourier coefficient, i.e.\begin{equation}\label{tau3a}
   \sum_{n \leq x}\tau_3(n)\lambda(n+1),
\end{equation}
Pitt \cite{Pi} established a power saving bound for the corresponding shifted convolution problem, which is an important ingredient in his cuspidal version of the Titchmarsh divisor problem \cite{Pi2}. %Pitt's result has been improved by Munshi \cite{Mu2} and the 
The current record $O(x^{5/6+\theta/3 + \varepsilon})$ for a smooth version is due to H. Tang \cite{Ta}, using ideas from \cite{To}.  On the other hand, when both factors in \eqref{tau3tau} are cuspidal, i.e.
\begin{equation}\label{cuspidal}
\sum_{n \leq x} A(n, 1)\lambda(n+1)
\end{equation}
  for a ${\rm GL}(3)$ Fourier coefficient $A(n, 1)$ and a ${\rm GL}(2)$ Fourier coefficient $\lambda(n)$, Munshi obtained a power saving bound; the current record for a smooth version is  $O(x^{21/22 + \varepsilon})$ due to P.\ Xi \cite{Xi}. 
  
  One may argue that this is the hardest case, since none of the two arithmetic functions can be decomposed as a convolution of simpler functions, but this feature is only one aspect in the analysis. Munshi's proof of \eqref{cuspidal} uses Jutila's very flexible version of the circle method, which is only (directly) applicable if general exponential sums in at least one of the involved arithmetic functions have uniform square-root cancellation. This is not true for the divisor function and not known for ${\rm GL}(3)$ Hecke eigenvalues. 
  % if none of the two arithmetic functions produces a main term or when squareroot cancellation on \emph{all} arcs is available.  
  In particular, the last remaining case 
  \begin{equation}\label{last} 
\sum_{n \leq x} A(n, 1)\tau(n+1)
\end{equation}
remained open and cannot be attacked by any of the methods used to treat \eqref{tau3tau}, \eqref{tau3a} or \eqref{cuspidal}.

 In this paper we solve this case, with  a more general  shift condition and complete uniformity in the bilinear shifting equation. 

\begin{theorem}\label{thm1} Let $  h, \lambda_1, \lambda_2 \in \Bbb{Z} \setminus \{0\}$,   $x \geq 1$. Let $W, W_0$ be  smooth functions with compact support in $[1, 2]$. Let $A(n, 1)$ denote the Hecke eigenvalues of a cusp form $F$ for the group ${\rm SL}_3(\Bbb{Z})$. Then
$$\sum_{\lambda_1 m- \lambda_2n = h} A(n, 1) \tau(m) W_0\Big(\frac{|\lambda_1| m}{x}\Big)  W\Big(\frac{|\lambda_2| n}{x}\Big)\ll_{F, W, W_0,  \varepsilon} x^{41/42 + \varepsilon}$$
 for any $\varepsilon > 0$, uniformly in $h, \lambda_1, \lambda_2$. 
\end{theorem}

While the result is uniform in $\lambda_1, \lambda_2$, we think of these coefficients as essentially fixed. If necessary, one can obtain additional small savings in $\lambda_1, \lambda_2$ (since the summation range becomes shorter), but we did not pursue this further. 

With applications in mind, we also prove a slightly more flexible variation. For $A, B \geq 1$ and two functions $v_1, v_2$ (suppressed from the notation) let 
$$\tau_{A, B}(n) := \sum_{ab = n} v_1\Big( \frac{a}{A}\Big) v_2\Big( \frac{b}{B}\Big). $$

\begin{theorem}\label{thm1a} Let $  h, \lambda_1, \lambda_2 \in \Bbb{Z} \setminus 
 \{0\}$ and $x, A, B \geq 1$ such that $AB \asymp x/|\lambda_1|$. Let $W, v_1, v_2$ be   smooth functions with compact support in $[1, 2]$. Let $A(n, 1)$ denote the Hecke eigenvalues of a cusp form $F$ for the group ${\rm SL}_3(\Bbb{Z})$. Then
$$\sum_{\lambda_1 m- \lambda_2n = h} A(n, 1) \tau_{A, B}(m)   W\Big(\frac{|\lambda_2| n}{x}\Big)\ll_{F, W,  v_1, v_2,  \varepsilon} x^{41/42+\varepsilon}$$
 for any $\varepsilon > 0$, uniformly in $A, B, h, \lambda_1, \lambda_2$. 
\end{theorem}

The proofs of Theorems \ref{thm1} and \ref{thm1a} combine for the first time two different delta symbol methods -- Jutila's method and a modern version of the Kloosterman method -- that are applied in an intertwined fashion. Jutila's method gives the flexibility to choose moduli in a way that creates a bilinear structure, but it only approximates a delta-function in an $L^2$-sense. On the other hand, exponential sums with divisor functions behave badly in an $L^2$-sense, since they  become very large on major arcs. Thus we invoke a second circle method  to have a tool that  is sensitive to the behaviour of these exponential sums. %The estimation of the minor arcs is relatively straightforward. 
On the major arcs, the key observation is that an extra Kloosterman refinement is possible, i.e.\ a non-trivial (and in fact square-root saving) estimate over the fractions $b/c$ for $b$ modulo $c$. That this is possible is not obvious a priori, but depends on the interplay of the two circle methods. We import Munshi's idea \cite{Mu1} to choose the moduli in Jutila's method in a factorable way to create a bilinear structure. However, our arrangement of Poisson, Voronoi and Cauchy--Schwarz steps differs from all other previous treatments  of ${\rm GL}(3) \times {\rm GL}(2)$ shifted convolution sums. 
%It also requires us to choose the moduli differently than Munshi in his cuspidal version of the shifted convolution problem. 

\subsection{An application} That the problem \eqref{last} is not an artificial construction may be supported by the following application that establishes an asymptotic formula for a twisted moment of $L$-functions on ${\rm GL}(3)$.

\begin{theorem}\label{thm2}  Let $Q \geq 1$ and let $F$ be a cusp form   for the group ${\rm SL}_3(\Bbb{Z})$.  Let $W$ be a smooth function with compact support in $[1, 2]$ and Mellin transform $\widetilde{W}$. Then
$$\sum_{q} W\Big(\frac{q}{Q}\Big) \sum_{\substack{\chi\, (\text{{\rm mod }} q)\\ \chi \text{ {\rm  primitive, even}}}} L(1/2, F \times \chi) = \frac{\widetilde {W}(2)}{2 \zeta(2)^2} Q^2 + O_{F, W,  \varepsilon}(Q^{2-1/41+\varepsilon})$$
for any $\varepsilon > 0$. 
\end{theorem}

A similar formula can be obtained when averaging over odd primitive characters. %The power saving exponent can be improved slightly at the cost of more work. 
 Theorem \ref{thm2} features a moment containing roughly $Q^2$ terms for an $L$-function of conductor roughly $Q^3$. Nevertheless, up until now, only a \emph{lower} bound was avaliable \cite{Lu} which was a hard-earned result and is now over 20 years old. The connection of this moment to \eqref{last} comes from a divisor-switching trick. We take an unbalanced approximate functional equation, where the first term has length $Q^{2+\delta}$ and the root number term has length $Q^{1-\delta}$ for some very small $\delta > 0$. Then the root number term can be estimated trivially (and as long as we cannot average non-trivially hyper-Kloosterman sums over the modulus, we don't have any better tools available). Applying orthonality of characters, we are left with
$$\sum_{q \asymp Q} \sum_{\substack{n \asymp Q^{2 + \delta}\\ n \equiv 1 \, (\text{mod } q)}} A(n, 1) \approx \sum_{q \asymp Q} \sum_{r \asymp Q^{1+\delta}} A(1 + rq, 1)$$
and the connection to \eqref{last} becomes clear. 

As an aside we remark that also Luo's result \cite{Lu} used crucially the idea of factorable moduli, and that Theorems \ref{thm1}, \ref{thm1a}  and \ref{thm2} become relatively straightforward if non-trivial averages of   hyper-Kloosterman sums over the modulus were available.

  \section{Preparation}

We will generally use the following standard conventions: the value of $\varepsilon$ can change from line to line (any typically picks up divisor functions, logarithms etc.\ on the way), and we write $a \mid b^{\infty}$ to mean that all prime divisors of $a$ divide $b$. Similarly,  $(a, b^{\infty}) = \lim_{n \rightarrow \infty} (a, b^n)$. 
  
 \subsection{Delta symbol methods}
 
 In this subsection we present two delta symbol methods. The first one is a very flexible method due to Jutila \cite{Ju}.  It gives, however, only an approximation to the constant function in an $L^2$-sense.
 
 \begin{lemma}\label{Jut} Let $Q\geq 1$, $\omega : [1, Q] \rightarrow [0, \infty)$, $L = \sum_{q} \phi(q) \omega(q)$ such that $L \not= 0$. Let $\psi : [-1, 1] \rightarrow [0, 1]$ be a smooth   function  with $\int \psi = 1$ and $0 < \delta < 1/2$. For $\alpha \in \Bbb{R}$ define the 1-periodic function
 $$\chi(\alpha) =\frac{1}{  \delta L} \sum_{q} \omega(q) \sum_{\substack{a \, (\text{{\rm mod }} q)\\ (a, q) = 1}} \sum_{k \in \Bbb{Z}}\psi\Big( \frac{1}{\delta}\Big(\alpha - \frac{a}{q} + k\Big)\Big).$$
 Then
 $$\int_0^1 (1 - \chi(\alpha))^2 d\alpha \ll_{\psi} \frac{Q^2 \| w \|_{\infty} |\log \delta |^3}{L^2 \delta}.$$
 \end{lemma}
Indeed, the $\ell$-th Fourier coefficient of $\chi$ equals
\begin{equation}\label{Four}
\frac{1}{  L} \sum_{q} \omega(q) r_q(\ell)  \hat{\psi}(\delta \ell) \begin{cases} = 1, & \ell = 0,\\\ll_{\psi} \frac{Q}{L}  (1 + \delta |\ell|)^{-10}  \tau(|\ell|)  \| \omega \|_{\infty} , & \ell \not = 0,\end{cases}
\end{equation}
where $\hat{\psi}$ denotes the Fourier transform and $r_q(\ell)$ the Ramanujan sum. The claim follows easily from Parseval. 

We derive the following useful representation for $\alpha = \frac{b}{c} + z$ where $b, c \in \Bbb{Z}$, $c \not= 0$ and $z \in \Bbb{R}$. Opening the Ramanujan sum and applying Poisson summation we have
\begin{equation}\label{repchi}
\begin{split}
 \chi\Big( \frac{b}{c} + z\Big) & %= \frac{1}{L} \sum_{\ell  } \hat{\psi}(\delta \ell) e\Big(- \Big(\frac{b}{c} + z\Big) \ell\Big)\sum_q \omega(q) r_q(\ell)\\
 = \frac{1}{L} \sum_q \omega(q) \sum_{d\mid q} d\mu\Big(\frac{q}{d}\Big) \sum_{\ell } \hat{\psi}(\delta d\ell) e\Big(- \Big(\frac{b}{c} + z\Big) \ell\Big)\\
 &  =\frac{1}{\delta L} \sum_q \omega(q) \sum_{d\mid q} \mu\Big(\frac{q}{d}\Big) \sum_{\ell \equiv b d\, (\text{mod } c)}\psi\Big( \frac{1}{\delta} \Big( \frac{\ell}{cd} + z\Big)\Big).
\end{split}
\end{equation}

%See \cite[Section 4]{Pi2} for a proof of this version.  

The second delta symbol is a version of a Kloosterman refinement of the circle method in the style of Heath-Brown \cite[Section 3]{HB}. 
\begin{lemma}\label{Farey}
	Let $C\geq 1$ and $n\in \mathbb Z$. Then 
	\begin{align}\label{bcz}
		\delta_{n=0}&=\sum_{c\leq C}\sum_{\substack{b\ppmod c\\ (b, c)=1}}\int_{-\frac{1}{c(c+c')}}^{\frac{1}{c(c+c'')}} e\Big(\Big(\frac{b}{c}+z\Big)n\Big)dz
	\end{align}
	where $\frac{b'}{c'}, \frac{b}{c}, \frac{b''}{c''}$ are consecutive Farey fractions of level $C$. We have 
	\begin{equation}\label{Fareycb}
		%\sum_{\substack{b\ppmod c\\ (b, c)=1}}
		\int_{-\frac{1}{c(c+c')}}^{\frac{1}{c(c+c'')}} e\Big(\Big(\frac{b}{c}+z\Big)n\Big)dz=\int_{-1/cC}^{1/cC}\frac{1}{c}\sum_{u\ppmod c}\sum_{t\in I(c,z)}e\Big(\frac{ut}{c}\Big) %\sum_{\substack{b\ppmod c\\ (b,c)=1}}
		e\Big(\frac{u\bar{b}}{c}\Big)e\Big( \Big(\frac{b}{c}+z\Big)n\Big)dz
	\end{equation}
	where
\begin{equation}\label{defI}
I(c,z)=%\begin{cases}
	%(C-c,\max(\frac{1}{c|z|}-c,C)], & \text{ if } z <0,\\
	\Big(C-c, \max\Big(\frac{1}{c|z|}-c, C\Big)\Big]. %, & \text{ if } z>0.
%\end{cases}
\end{equation}
\end{lemma} 
\begin{proof}
We decompose the interval $[0,1]$ using Farey fractions of level $C$, so that 
\begin{align}
	&\int_0^1e(\alpha n)d\alpha=\sum_{c\leq C}\sum_{\substack{b\ppmod c\\ (b, c) = 1}}\int_{\frac{b+b'}{c+c'}-\frac{b}{c}}^{\frac{b+b''}{c+c''}-\frac{b}{c}}e\Big(\Big(\frac{b}{c}+z\Big)n\Big)dz,
%	&=\sum_{c\leq C}\sum_{b\ppmod c}^* \int_{-\frac{1}{c(c+c')}}^{\frac{1}{c(c+c'')}} f(\frac{b}{c}+z)dz
\end{align}
where $\frac{b'}{c'}, \frac{b}{c}, \frac{b''}{c''}$ are consecutive Farey fractions. Since $b''c-bc''=bc'-b'c=1$, the integral runs over the interval $[-(c(c+c'))^{-1}, (c(c+c''))^{-1}]$ and we obtain the first equality.  

From the conditions $c'\equiv -c''\equiv \bar{b}\pmod c$ and $C-c<c',c''\leq C$, we see that there is a unique pair $(c', c'')$ which determines $\bar{b}  \pmod c$. Since $c+c', c+c''>C$, we can write 
\begin{align}
	%\sum_{\substack{b\ppmod c\\ (b,c)=1}}	
	\int_{-\frac{1}{c(c+c')}}^{\frac{1}{c(c+c'')}}e\Big(\Big(\frac{b}{c}+z\Big)n\Big)dz=\int_{-\frac{1}{cC}}^{\frac{1}{cC}}%\sum_{\substack{b\ppmod c\\(b,c)=1\\ \bar{b}\equiv t\ppmod c, \,  t\in I(c,z)}} 
	\textbf{1}_{\bar{b}\equiv t\ppmod c \text{ for some }   t\in I(c,z)}	e\Big(\Big(\frac{b}{c}+z\Big)n\Big)dz
\end{align} 
where $I(c, z)$ is as in \eqref{defI}. Detecting the congruence $\bar{b} \equiv t$ (mod $c$) with additive characters gives the second equality. 
%\begin{align}
%	I(c,z)=\begin{cases}
%		(C-c,\max(\frac{1}{c|z|}-c,C)] & z<0,\\
%		(C-c, \max(\frac{1}{c|z|}-c, C)] & z>0.
%	\end{cases}
%\end{align}
%If $|z|< \frac{1}{c(c+C)}$, then there is no restriction on $\bar{b}$. Otherwise we can detect the condition on $\bar{b}$ using additive characters modulo $c$ so that
%\begin{align}
%	\sum_{\substack{b\ppmod c\\ (b, c) = 1\\ \bar{b}\equiv t\ppmod c, \, t\in I(c,z)}}e\Big(\Big(\frac{b}{c}+z\Big)n\Big)&=\sum_{t\in I(c,z)}\sum_{\substack{b\ppmod c\\ (b, c) = 1}}\frac{1}{c}\sum_{u\ppmod c}e(\frac{u(\bar{b}-t)}{c})e\Big(\Big(\frac{b}{c}+z\Big)n\Big)\\
%	&=\frac{1}{c}\sum_{t\in I(c,z)}e\Big(\frac{ut}{c}\Big) \sum_{\substack{b\ppmod c\\(b,c)=1}} e\Big(\frac{u\bar{b}}{c}\Big)e\Big(\Big(\frac{b}{c}+z\Big)n\Big)
%\end{align}
%which gives the second equality. 
%Thus 
%\begin{align}
%	\sum_{b\ppmod c}^*	\int_{-\frac{1}{c(c+c')}}^{\frac{1}{c(c+'')}}f(\frac{b}{c}+z)dz=&\int_{-\frac{1}{c(c+C)}}^{\frac{1}{c(c+C)}}\sum_{b\ppmod c}^*f(\frac{b}{c}+z)dz\\
%	&+\int_{-\frac{1}{cC}}^{-\frac{1}{c(c+C)}}\frac{1}{c}\sum_{u\ppmod c}\sum_{t\in I(c,z)}e_{c}(ut)\sum_{b\ppmod c}^*e_{c}(u\bar{b})f(\frac{b}{c}+z)dz\\
%	&+\int^{\frac{1}{cC}}_{\frac{1}{c(c+C)}}\frac{1}{c}\sum_{u\ppmod c}\sum_{t\in I(c,z)}e_{c}(ut)\sum_{b\ppmod c}^*e_{c}(u\bar{b})f(\frac{b}{c}+z)dz
%\end{align} 
\end{proof} 

%The second delta-symbol method is a classical delta-symbol method in the style of Duke-Friedlander-Iwaniec; see \cite[Proposition 1.2]{MV}.
%\begin{lemma} Let $C \geq 1$ and $n \in \Bbb{Z}$. Then 
%$$\delta_{n=0} = \sum_{c \leq C} \sum_{\substack{b \, (\text{{\rm mod }} c)\\ (b, c) = 1}} \int_{|z| \ll (cC)^{-1+\varepsilon}} p_c(z) e\Big(\Big( \frac{b}{c} + z\Big)n\Big)dz + O(C^{-A})$$
%for any $A, \varepsilon > 0$ and some function $p_c(z) \ll 1$. 
%\end{lemma}

With the same notation we conclude for a smooth, one-periodic  function $f$  (by decomposing into its Fourier series) that 
%with Fourier coefficients $\hat{f}(n)$ that 
\begin{equation}\label{MaVi}
\int_0^1 f(z) dz  = \sum_{c \leq C}\int_{-1/cC}^{1/cC}\frac{1}{c}\sum_{u\ppmod c}\sum_{t\in I(c,z)}e\Big(\frac{ut}{c}\Big) \sum_{\substack{b \, (\text{{\rm mod }} c)\\ (b, c) = 1}} e\Big(\frac{u\bar{b}}{c}\Big)f \Big( \frac{b}{c} + z\Big)dz. 
\end{equation}

 \subsection{Voronoi summation} 

The following  two Voronoi summation formulae are well-known. 

\begin{lemma}\label{vortau} Let $c\in \Bbb{N}$, $b\in \Bbb{Z}$, $(b, c) = 1$, $w$ a smooth  function with compact support in $(0, \infty)$. Then
$$\sum_n \tau(n) e\Big( \frac{b}{c} n\Big) w(n) = \frac{1}{c} \int_0^{\infty} w(\xi) \Big(\log \frac{\xi}{c^2} + 2\gamma\Big) \, d\xi +  \frac{1}{c}  \sum_{\pm}  \sum_{n} \tau(n) e\Big( \pm  \frac{\bar{b}}{c} n\Big) \int_0^{\infty} w(\xi) J^{\pm} \Big(  \frac{\sqrt{n\xi}}{c}\Big) d\xi$$
where $J^{-} (\xi) =  \sum_{\pm}e(\pm 2\xi) v_{\pm}(\xi)$, $J^{+}(\xi) = v_0(\xi)$ with
$$\xi^j v^{(j)}_{\pm}(\xi) \ll_j \frac{1+ |\log \xi|}{1 + \xi^{1/2}}, \quad v_0(\xi) \ll_A \frac{1+|\log \xi|}{1 + \xi^A}$$
for any $j, A \geq 0$.  
\end{lemma}

For future reference, we analyze the integral transform in the case in the following special case. 

\begin{lemma}\label{specialtau} Let $X \geq 1$, $|Z| \leq 1$ and $W$ be a   fixed smooth function with support in $[1, 2]$. Let $n, c \in\Bbb{N}$.   For 
$$w(\xi) = w_{X, Z}(\xi) = W\Big(\frac{\xi}{X}\Big)e(\xi Z)$$
we have 
\begin{displaymath}
\begin{split}
& \int_0^{\infty} w(\xi) J^{-} \Big(  \frac{\sqrt{n\xi}}{c}\Big) d\xi \ll_{A, \varepsilon} \frac{X (Xnc)^{\varepsilon}}{1 + X|Z|} \Big(1 + \frac{nX}{c^2(1 + X|Z|)^2}\Big)^{-A},\\
&  \int_0^{\infty} w(\xi) J^{+} \Big(  \frac{\sqrt{n\xi}}{c}\Big) d\xi \ll_{A, \varepsilon} X(Xnc)^{\varepsilon}  \Big(1 + \frac{nX}{c^2 }\Big)^{-A}
 \end{split}
 \end{displaymath}
 for all $\varepsilon, A > 0$. 
\end{lemma}

\begin{proof} Put $P = Xnc$. The second bound follows by trivial estimation. As long as   $|Z| \leq P^{\varepsilon}/x$, the first bound follows by a simple integration by  parts argument. 

 Let us now consider the first bound when $|Z| \geq P^{\varepsilon}/X$. In this case we are looking at 
$$\int_0^{\infty} W\Big(\frac{\xi}{X}\Big) v_{\pm}\Big(    \frac{ \sqrt{n\xi}}{c}\Big) e\Big(\xi Z \pm   \frac{2\sqrt{n\xi}}{c}\Big)d\xi.$$
There is at most one stationary point at $\xi = n(cZ)^{-2}$. If $X \not\asymp n(cZ)^{-2}$ we apply integration by parts in the form of \cite[Lemma 8.1]{BKY} with
$${\tt U} = {\tt Q} = (\beta - \alpha) = X, \quad {\tt R} = |Z| + \frac{n^{1/2}}{X^{1/2}c}, \quad {\tt X} = \frac{1 + |\log \frac{nX}{c^2}|}{1 +  (nXc^{-2})^{1/4}}, \quad {\tt Y} =  \frac{\sqrt{nX}}{c}$$
to bound the integral by
\begin{displaymath}
\begin{split}
&\ll _A x\frac{1 + |\log \frac{nX}{c^2}|}{1 +  (nXc^{-2})^{1/4}} \Big[ \Big(\Big( X|Z| + \frac{\sqrt{nX}}{c}\Big) \Big(\frac{\sqrt{nX}}{c}\Big)^{-1/2}\Big)^{-A}+  \Big( X|Z| + \frac{\sqrt{nX}}{c}\Big)^{-A}\Big]\\
& \ll_A X\frac{ (Xnc)^{\varepsilon}}{1 +  (nXc^{-2})^{1/4}}   \Big( X|Z| + \frac{\sqrt{nX}}{c}\Big)^{-A/2} 
\end{split}
\end{displaymath}
which is stronger than claimed (after changing the constant $A$). 

Assume now $x \asymp n(cz)^{-2}$ in which case the target bound is $P^{\varepsilon}|Z|^{-1}$. %If $ then we apply \cite[Proposition 8.2]{BKY} with
Put
$${\tt V} = {\tt V_1} = {\tt Q} = X, \quad {\tt Y} =  \frac{\sqrt{nX}}{c}, \quad  {\tt X} = \frac{1 + |\log \frac{nX}{c^2}|}{1 +  (nXc^{-2})^{1/4}}.$$
Since ${\tt Y} \asymp X |Z| \gg P^{\varepsilon}$,  we can apply \cite[Proposition 8.2]{BKY} to obtain the bound $$\ll \frac{ {\tt X}}{\sqrt{{\tt Y} {\tt Q}^{-2}}} \ll P^{\varepsilon}|Z|^{-1}.$$This completes the proof. 
\end{proof}
 
\begin{lemma}\label{vorGL3}  Let $c\in \Bbb{N}$, $b\in \Bbb{Z}$, $(b, c) = 1$, $w$ a smooth  function with compact support in $(0, \infty)$.  Then
$$\sum_{n} A(n, 1) e\Big( \frac{b}{c}n\Big) w(n) = \frac{1}{c^2} \sum_{\pm} \sum_{n_2} \sum_{n_1 \mid c} n_1 A(n_1, n_2) S\Big(\bar{b}, \pm n_2, \frac{c}{n_1}\Big) \int_0^{\infty} w(y) V^{\pm}\Big(\frac{n_1^2n_2 y}{c^3}\Big) dy$$
where
$$V^{\pm}(\xi) = R(\xi) \frac{e(\pm 3 \xi^{1/3})}{\xi^{1/3}} + \frac{S(\xi)}{\xi^{1/2}}$$
with 
$$\xi^k \frac{d^k}{d\xi^k} R(\xi) \ll_k \bm{1}_{\xi \gg 1}, \quad \xi^k \frac{d^k}{d\xi^k}  S(\xi)\ll_k \bm{1}_{\xi \ll 1}.$$ 
\end{lemma} 

Indeed, the function $V^{\pm}$ is the inverse Mellin transform of
$$G^{\pm}(s) = \frac{\pi^{3/2-3s}}{2} \Big(\prod_{j=1}^3 \frac{\Gamma (\frac{s + \alpha_j}{2} ) }{\Gamma (\frac{1-s - \alpha_j}{2} )} \pm \frac{1}{i}\prod_{j=1}^3 \frac{\Gamma (\frac{s+1 + \alpha_j}{2} ) }{\Gamma (\frac{2-s - \alpha_j}{2} )}\Big),$$
where $\{\alpha_1, \alpha_2, \alpha_3\}$ is the Langlands parameter of the underlying cusp form $F$. 
The bound for $\xi \ll 1$ follows from shifting the contour to the left using that $\max_j |\Re \alpha_j| < 1/2$, while the bound for $\xi \gg 1$ follows from \cite[Lemma 6]{Bl}. 

\medskip

As an analogue of Lemma \ref{vortau} we state the following structurally similar formula for a convolution, which follows easily from two applications of Poisson summation.

\begin{lemma}\label{vorconv} Let  $w$ be a smooth function with compact support in $(0, \infty)^2$, $c\in \Bbb{N}$ and $b \in \Bbb{Z}$ with $(b, c) = 1$. Then
$$\sum_{r, q}   w(r, q) e\Big(\frac{rqb}{c}\Big) =  \frac{1}{c} \sum_{r, q \in \Bbb{Z}} e\Big(-\frac{rq\bar{b}}{c}\Big) \int_{\Bbb{R}^2} w(x, y) e\Big(\frac{xr + yq}{c}\Big) dx\, dy.$$
\end{lemma}

\subsection{Bounds for Hecke eigenvalues}

We will frequently use the following bounds which follow from   the Hecke relation
$A(n_1, n_2) = \sum_{d\mid(n_1, n_2)} \mu(d) A(n_1/d, 1)A(1, n_2/d)$, Rankin-Selberg theory and trivial bounds towards the Ramanujan conjecture: we have
\begin{equation}\label{hecke}
\sum_{n \asymp X} |A(n, m)| \ll_{\varepsilon} X m^{1/2  } (Xm)^{\varepsilon}, \quad \sum_{n \asymp X} \sum_{m \asymp Y} |A(n, m)|  \ll_{\varepsilon} (XY)^{1+\varepsilon}
\end{equation}
for $X, Y \geq 1$, $\varepsilon > 0$.

\subsection{Character sums}
Both the delta symbol methods and the Voronoi summation formulae create various character sums for which it is important to have best possible bounds -- at least in generic cases -- uniformly in several auxiliary parameters. 

For $c \in \Bbb{N}$, $n_1 \mid c$, $h,  d, n_2 \in \Bbb{Z}$ we define the character sum
\begin{equation}\label{Sigma}
\Sigma_{h,  d,n_1, n_2}(c)  := \sum_{\substack{b\, (\text{mod } c) \\ (b, c) = 1}} e\Big(  \frac{bh+\bar{b}d}{c}\Big)  S\Big(\bar{b}, n_2, \frac{c}{n_1}\Big).
\end{equation}

\begin{lemma}\label{char1}
	Let $n_1\mid c, h,  d, n_2\in \mathbb Z$.  
	We decompose uniquely  $c=c_1c_2$ with $c_1$ squarefree, $c_2$ squarefull and $(c_1, c_2) = 1$. 
	Then \begin{align}
	\Sigma_{h, d,n_1, n_2}(c) %\sum_{\substack{b\ppmod c\\ (b,c)=1}}e\Big(\frac{hb+\bar{b}\bar{k}d}{c}\Big)S(\bar{b}, n_2, \frac{c}{n_1})
	\ll_{\varepsilon}  c^{1+\varepsilon}c_{2}^{1/2}\frac{\sqrt{(n_1, c_1, d,h)}}{\sqrt{n_1}}
	\leq c^{1+\varepsilon}c_{2}^{1/2}
	\end{align}
	for any $\varepsilon > 0$. 
\end{lemma}

\begin{remark} With more work it should be possible to remove the factor $c_{2}^{1/2}$, at least in typical cases. The present bound suffices for our purposes. %The proof shows that we obtain slightly stronger results if $n_1, n_2
\end{remark}

\begin{proof}
	We first consider the case when $c=p$ is prime. If $p\mid n_1$, then the sum becomes a Kloosterman sum and 
\begin{align}
\Sigma_{h, d,n_1, n_2}(p) =\sum_{\substack{b\ppmod p\\ (b,p)=1}}e\Big(\frac{bh+\bar{b} d}{p}\Big) \ll \sqrt{p}\sqrt{(p,h,d)}.
\end{align}
If $p\nmid n_1$ and $p\mid n_2$, then 
\begin{align}
	\Sigma_{h,d, n_1, n_2}(p)=\sum_{\substack{b,x\ppmod c\\ (bx,c)=1}}e\Big(\frac{hb+\bar{b} d+\bar{b}xn_1}{c}\Big)\ll \sqrt{p}\sqrt{(p,h,d)}. 
\end{align}
If $p\nmid n_1n_2$, $p\mid h$, then 
\begin{align}
\Sigma_{h ,d,n_1, n_2}(p)  =\sum_{\substack{b, x\ppmod p\\ (bx,p)=1 }}e\Big(\frac{\bar{b} d+\bar{b}n_1x+n_2n_1\bar{x}}{p}\Big)\ll \sum_{\substack{x\ppmod p\\ (x,p)=1}} (d+n_1x, p)\ll p. 
\end{align}
If $p\nmid n_1n_2h$, then 
\begin{align}
\Sigma_{h, d,n_1, n_2}(p)  =\sum_{\substack{b, x\ppmod p\\ (bx,p)=1 }}e\Big(\frac{hb+\bar{b} (d+ n_1x)+n_2n_1\bar{x}}{p}\Big)=\sum_{\substack{b, x\ppmod p\\ (bx,p)=1 }}e\Big(\frac{b+\bar{b}h (d+ xn_1^2n_2)+\bar{x}}{p}\Big) \ll p
\end{align}
by the bounds of  Adolphson--Sperber \cite{AS} with the Newton polygon $\{(1, 0),  (-1, 0), (-1, 1), (0, -1)\}$ if $p \nmid d$, and by Deligne's bound for hyper-Kloosterman sums if $p\mid d$. The desired bound (without the factor $c_{2}^{1/2}$) follows now from the Chinese remainder theorem if $c$ is squarefree. 

On the other hand, we always have 
\begin{align}
	\Sigma_{h, d,n_1n_2}(c)&\ll \sum_{\substack{b\ppmod c\\ (b,c)=1}}\Big|S\Big(\bar{b}, n_2, \frac{c}{n_1}\Big)\Big|\ll c\sqrt{\frac{c}{n_1}}
\end{align}
which again by the Chinese remainder theorem concludes the proof in all cases. 
\end{proof}
We will also need to estimate the character sum %for $(n_1, p_1p_2)=1$. 
 \begin{equation}\label{defT}
	\mathcal T(h,  d_1, d_2, n_1, n_2, p_1, p_2, t):=\sum_{x\ppmod {[p_1, p_2]t}}e\Big(\frac{  xn_2}{[p_1, p_2]t}\Big)\Sigma_{h,   d_1, n_1,  x}(p_1t)\overline{\Sigma_{h,   d_2, n_1, x}(p_2t)}
\end{equation}
for two primes $p_1, p_2$ and $n_1 \mid t$. The precedent of the following lemma is \cite[Lemma 10 \& 11]{Mu1} which however requires some corrections as noted in \cite{Xi}. Our version needs more refined bounds. 

\begin{lemma}\label{charT}
	Let $p_1, p_2$ be two primes, $n_1 \mid t = t_1t_2$ with $t_1$ squarefree, $t_2$ squarefull, $(t_1, t_2) = 1$. Assume $(hn_1, p_1p_2)=1$ and let $\varepsilon > 0$. 
	
	If $p_1\not=p_2$, then $\mathcal T$ vanishes unless $(n_2, p_1p_2)=1$ in which case we have% (writing $t=t_1t_2$ where $t_1$ is squarefree and $t_2$ is the squarefull part of $t$)
	\begin{equation}
		\mathcal T(h, d_1, d_2, n_1, n_2, p_1, p_2, t)\ll_{\varepsilon} p_1^{3/2}p_2^{3/2}t^{5/2+\varepsilon}(hn_2, t_1)^{1/2}t_2^{1/2}.
	\end{equation}
	If $p_1=p_2$, then we have 
	\begin{equation}\label{Tbound}
		\mathcal T(h,  d_1, d_2, n_1, n_2, p_1, p_1, t)\ll_{\varepsilon} p_1^{3}t^{5/2+\varepsilon}(hn_2p_1, t_1)^{1/2}t_2^{1/2}.
	\end{equation}
	If $p_1=p_2$ and $n_2=0$, then we have an improved bound
	\begin{align}
		\mathcal T(h, d_1, d_2, n_1, n_2, p_1, p_1, t)\ll_{\varepsilon} p_1^{2}t^{3+\varepsilon} (d_1-d_2, p_1).
	\end{align}
\end{lemma}
\begin{proof}
Consider the case $p_1\not=p_2$ first. We write $t=g_1g_2\tau_1\tau_2$ where $g_1\mid p_1^\infty, g_2\mid p_2^\infty, (\tau_1\tau_2, p_1p_2)=1$, $\tau_1$ squarefree and  $\tau_2$ is squarefull with $(\tau_1, \tau_2) = 1$.  Then $\mathcal T$ factors into a product of exponential sums modulo $p_1g_1, p_2g_2, \tau_1, \tau_2$. If $g \mid p_1p_2t$,  we will generally use the notation $g' = p_1p_2t/g$ for the co-divisor. 

%Denote $g'=p_1p_2t/g$ for $g\mid t$. Then 
The sum modulo $p_1g_1$ is given by 
\begin{align}
	\sum_{x\ppmod {p_1g_1}}e\Big(\frac{xn_2\overline{(p_1g_1)'}}{p_1g_1}\Big)\sum_{\substack{b\ppmod {p_1g_1}\\ (b, p_1)=1}}e\Big(\frac{(bh+\bar{b} d_1)\overline{(p_1g_1)'}p_2}{p_1g_1}\Big)S(\bar{b}, xn_1^2 \overline{(p_1g_1)'}^2p_2^2, p_1g_1).
\end{align}
If $p_1 \mid n_2$, the sum over $x$ vanishes since $(n_1, p_1) = 1$. Otherwise, we open the Kloosterman sum, sum over $x$ and  obtain
\begin{align}
	p_1g_1 \sum_{\substack{b\ppmod {p_1g_1}\\ (b, p_1)=1}}e\Big(\frac{(bh+\bar{b}d_1)\overline{(p_1g_1)'}p_2-\overline{bn_2} n_1^2p_2^2\overline{(p_1g_1)'}}{p_1g_1}\Big)\ll (p_1g_1)^{3/2}
\end{align} 
since $(h, p_1)=1$ and $g_1\mid p_1^\infty$. 

The sum modulo $p_2g_2$ can be estimated in the same way. 

Recall that $n_1 \mid t$ and $(n_1, p_1p_2) = 1$ and decompose  $n_1=n_{11}n_{12}$ with  $n_{11}\mid \tau_1$ and $n_{12}\mid \tau_2$. Then the exponential sum modulo $\tau_1$ becomes 
\begin{align}\label{t1exp}
	&\sum_{x\ppmod {\tau_1}}e\Big(\frac{xn_2\overline{\tau_1'}}{\tau_1}\Big)\sum_{\substack{\substack{b_1\ppmod {\tau_1}\\ (b_1, \tau_1)=1}}}e\Big(\frac{(b_1h+\overline{b_1 }d_1)\overline{\tau_1'}p_2}{\tau_1}\Big)S\Big(\overline{b_1}, x\overline{\tau_1'}^2n_{12}^2p_2^2, \frac{\tau_1}{n_{11}}\Big)\\
	& \times \sum_{\substack{\substack{b_2\ppmod {\tau_1}\\ (b_2, \tau_1)=1}}}e\Big(\frac{-(b_2h+\overline{b_2 }d_2) \overline{\tau_1'}p_1}{\tau_1}\Big)S\Big(\overline{b_2}, x\overline{\tau_1'}^2n_{12}^2p_1^2, \frac{\tau_1}{n_{11}}\Big).
\end{align}
Recall that $\tau_1$ is squarefree. We apply the Chinese remainder theorem and consider the sum in \eqref{t1exp} modulo a prime $p \mid \tau_1$. 

If $p\mid n_{11}$, then  the exponential sum in \eqref{t1exp} becomes
\begin{displaymath}
\begin{split}
	\sum_{x\ppmod p}e\Big(\frac{xn_2\overline{p'}}{p}\Big)&\sum_{\substack{b_1\ppmod p\\ (b_1, p)=1}}e\Big(\frac{(b_1h+\overline{b_1 }d_1)\overline{p'}p_2}{p}\Big)\sum_{\substack{b_2\ppmod p\\ (b_2, p)=1}}e\Big(\frac{-(b_2h+\overline{b_2 }d_2)\overline{p'}p_1}{p}\Big) \ll p^3 \mathbf{1}_{p\mid n_2} 
	% \ll p^{5/2}(p, n_2)^{1/2}.
	\end{split}
\end{displaymath}
which satisfies the crude bound $p^{5/2}(p, n_2)^{1/2}$. 

If $p\nmid n_{11}$, then $n_1 =  n_{12}$ and the exponential sum in \eqref{t1exp} becomes
\begin{equation}\label{pn11coprime}
\begin{split}
	\sum_{x\ppmod p}e\Big(\frac{xn_2\overline{p'}}{p}\Big)&\sum_{\substack{b_1\ppmod p\\ (b_1, p)=1}}e\Big(\frac{(b_1h+\overline{b_1 }d_1)\overline{p'}p_2}{p}\Big)S(\overline{b_1}, x\overline{p'}^2n_1^2p_2^2, p)\\
	&\times \sum_{\substack{b_2\ppmod p\\ (b_2, p)=1}}e\Big(\frac{-(b_2h+\overline{b_2 }d_2)\overline{p'}p_1}{p}\Big)S(\overline{b_2}, x\overline{p'}^2n_1^2p_1^2, p).
	\end{split}
\end{equation}
Opening the Kloosterman sums and summing over $x$, we obtain
\begin{equation}\label{weobtain}
\begin{split}
	\sum_{\substack{y_1, y_2\ppmod p\\(y_1y_2,p)=1\\ p'n_2+(y_1p_2^2+y_2p_1^2)n_1^2\equiv 0\ppmod p}}&p\sum_{\substack{b_1\ppmod p\\ (b_1, p)=1}}e\Big(\frac{(b_1h+\overline{b_1 }d_1)\overline{p'}p_2+\overline{b_1y_1}}{p}\Big)\\
	&\times \sum_{\substack{b_2\ppmod p\\ (b_2, p)=1}}e\Big(\frac{-(b_2h+\overline{b_2 }d_2)\overline{p'}p_1+\overline{b_2y_2}}{p}\Big).
	\end{split}
\end{equation} 
If $p\mid h$, then \eqref{weobtain} becomes 
\begin{align}
	&\sum_{\substack{y_1, y_2\ppmod p\\(y_1y_2,p)=1\\ p'n_2+(y_1p_2^2+y_2p_1^2)n_1^2\equiv 0\ppmod p}}p\sum_{\substack{b_1\ppmod p\\ (b_1, p)=1}}e\Big(\frac{\overline{b_1 }d_1\overline{p'}p_2+\overline{b_1y_1}}{p}\Big)\sum_{\substack{b_2\ppmod p\\ (b_2, p)=1}}e\Big(\frac{-\overline{b_2 }d_2\overline{p'}p_1+\overline{b_2y_2}}{p}\Big)\\
	& \ll p\sum_{\substack{y_1 , y_2\ppmod p\\(y_1y_2,p)=1%\\ p_2' n_2+(y_1p_2^2+y_2p_1^2)n_1^2\equiv 0\ppmod p
	}} \big(y_1d_1p_2+ p', p\big) \big(y_2d_2p_1+p', p\big)
%	& \ll p^2 \Big(\sum_{y_1\ppmod p}(y_1d+kp_1g_1g_2\frac{t_1}{p},p)\Big)^{1/2}\Big(\sum_{y_2\ppmod p}(y_2d+kp_2g_1g_2\frac{t_1}{p},p)\Big)^{1/2}\\
 \ll p^3=p^{5/2}(p,h)^{1/2}.
\end{align}
If $p\nmid h$, we change variables and 
write $y_1=x, \ b_1=y p'\overline{hp_2},\  b_2=-zp'\overline{hp_1}$. Then \eqref{weobtain} becomes 
\begin{equation}\label{t1pbound}
	p\sum_{\substack{x,y,z\ppmod p\\ (xyz,p)=1\\ (n_1^2p_1^2x+p'n_2, p)=1}}e\Big(y+\frac{b(x)}{y}+z+\frac{c(x)}{z}\Big)
\end{equation}
where
\begin{equation}
	\begin{split}
		b(x)=\frac{hp_2}{ p'^2}\Big(\frac{ p'}{x}+d_1p_2\Big),\quad 
		c(x)=\frac{hp_1^2}{ p'^2}\Big(\frac{ p'p_1n_1^2}{n_1^2p_2^2x+p'n_2}+d_2\Big).
	\end{split}
\end{equation}
We have
$$\frac{c(x)}{b(x)} = \frac{d_2n_2p_1^2 p' x + n_1^2 p_1^3 p' x + d_2n_1^2 p_1^2p_2^2 x^2}{n_2 p_2 (p')^2 + d_1n_2 p_2^2 p' x + n_1^2 p_2^3 p' x + d_1n_1^2 p_2^4 x^2}$$
which is  the constant one rational function if and only if 
$$p\mid p_1^3-p_2^3,\, p\mid d_2p_1^2-d_1p_2^2, \, p\mid n_2$$
by our current assumption $p \nmid h n_1$. 
%Note that $b(x)=c(x)$ if and only if 
%\begin{align}
%	&\begin{cases}
%		h n_2p'p_2\equiv 0 \ppmod p, \\
%%		h(d_2n_2p_1^2-d_1n_2p_2^2+ n_1^2p_1^3p'- n_1^2p_2^3p')\equiv 0\ppmod p,\\
%		hn_1^2p_2^2(d_2p_1^2-d_1p_2^2)\equiv 0\ppmod p
%	\end{cases}\\
%	&\Leftrightarrow 
%	\begin{cases}
%		p\mid h \text{ or }\\
%		p\mid p_1^3-p_2^3,\, p\mid d_2p_1^2-d_1p_2^2, \, p\mid n_2, \, p\nmid h.
%	\end{cases}
%\end{align}
In particular, by Bombieri--Sperber \cite[Theorem 4]{BS} we conclude that  \eqref{t1pbound} is $\ll p^{5/2}$ if $p\nmid  n_2$.   Otherwise, we use Weil's bound for the $y,z$-sums and bound the $x$-sum trivially getting the estimate  $\ll p^3$ for  \eqref{t1pbound}.  We summarize the preceding discussion as
$$\sum_{x \ppmod {p}} (...) \ll p^{5/2} (p, hn_2)^{1/2}$$
in all cases for $p \mid \tau_1$.

Finally, for the exponential sum modulo $\tau_2$ we have
\begin{align}\label{t2exp}
&\sum_{x\ppmod {\tau_2}}e\Big(\frac{xn_2\overline{\tau_2'}}{\tau_2}\Big)\sum_{\substack{\substack{b_1\ppmod {\tau_2}\\ (b_1, \tau_2)=1}}}e\Big(\frac{(b_1h+\overline{b_1 }d_1)\overline{\tau_2'}p_2}{\tau_2}\Big)S\Big(\overline{b_1}n_{11}, x\overline{\tau_2'}^2n_{11}p_2^2, \frac{\tau_2}{n_{12}}\Big)\\
& \times \sum_{\substack{\substack{b_2\ppmod {\tau_2}\\ (b_2, \tau_2)=1}}}e\Big(\frac{-(b_2h+\overline{b_2 }d_2) \overline{\tau_2'}p_1}{\tau_2}\Big)S\Big(\overline{b_2}n_{11}, x\overline{\tau_2'}^2n_{11}p_1^2, \frac{\tau_2}{n_{12}}\Big).
\end{align}
Expanding the Kloosterman sum and then summing over $x$ gives 
\begin{align}
	\tau_2\frac{\phi(\tau_2/n_{12})^2}{\phi(\tau_2)^2}\sum_{\substack{y_1, y_2\ppmod {\tau_2}\\ \tau_2'n_2+(y_1p_2^2+y_2p_1^2)n_1\equiv 0\ppmod {\tau_2}}}&\sum_{\substack{b_1\ppmod {\tau_2}\\ (b_1, \tau_2)=1}}e\Big(\frac{(b_1h+\overline{b_1 }d_1)\overline{\tau_2'}p_2+\overline{b_1y_1}n_1}{\tau_2}\Big)\\
	& \times\sum_{\substack{b_2\ppmod {\tau_2}\\ (b_2, \tau_2)=1}}e\Big(\frac{-(b_2h+\overline{b_2 }d_2)\overline{\tau_2'}p_1+\overline{b_2y_2}n_1}{\tau_2}\Big),
\end{align}
which   by Weil's bound can be bounded by 
\begin{equation*}%\label{weil}
\begin{split}
&\frac{\tau_2^{2+\varepsilon}}{n_{12}^2}\sum_{\substack{y_1, y_2\ppmod {\tau_2}\\ \tau_2'n_2+(y_1p_2^2+y_2p_1^2)n_1\equiv 0\ppmod {\tau_2}}} \sqrt{(h,y_1d_1p_2+\tau_2'n_1, \tau_2)(h,y_2d_2p_1+\tau_2'n_1, \tau_2)}\\
&\ll \frac{\tau_2^{2+\varepsilon}}{n_{12}^2}\Big(\sum_{\substack{y_1, y_2\ppmod {\tau_2}\\ \tau_2'n_2+(y_1p_2^2+y_2p_1^2)n_1\equiv 0\ppmod {\tau_2}}} (y_1d_1p_2+\tau_2'n_1, \tau_2)\Big)^{1/2}
\\&\quad\quad \quad \quad \times 
\Big(\sum_{\substack{y_1, y_2\ppmod {\tau_2}\\ \tau_2'n_2+(y_1p_2^2+y_2p_1^2)n_1\equiv 0\ppmod {\tau_2}}} (y_2d_2p_1+\tau_2'n_1, \tau_2)\Big)^{1/2}.
\end{split}
\end{equation*}
By symmetry it suffices to analyze one of the parentheses, say the first. The congruence determines $y_2$ modulo $\tau_2/(\tau_2, n_1)$. For a given value $\tau = (y_1 d_1p_2 + \tau_2'n_1, \tau_2) \mid \tau_2$ there are at most $\tau_2(d_1, n_1, \tau)/\tau$ choices for $y_1$, so that we can bound the previous display by 
\begin{align}
& \ll \frac{\tau_2^{2+\varepsilon}}{n_{12}^2}\tau_2^{1+\varepsilon}(n_1, \tau_2)^2\ll \tau_2^{3+\varepsilon},
\end{align}
noting that $(\tau_2, n_1) = (\tau_2, n_{12})$. 

Combining all previous estimates, we have a final bound 
\begin{align}
	(p_1g_1)^{3/2}(p_2g_2)^{3/2}\tau_1^{5/2+\varepsilon}(hn_2, \tau_1)^{1/2}\tau_2^{3+\varepsilon }\ll p_1^{3/2}p_2^{3/2}t^{5/2+\varepsilon}(hn_2, t_1)^{1/2}t_2^{1/2}
\end{align}
for $\mathcal{T}$ in the case $p_1 \not= p_2$. 

\medskip

Next we consider $p_1=p_2$. We write $t=g_1\tau_1\tau_2$ where $g_1\mid p_1^\infty$ and $(\tau_1\tau_2, p_1)=1$ and $\tau_1, \tau_2$ have the same meaning as before. Note that we have $(n_1, p_1g_1)=1$. 
 Modulo $p_1g_1$, the exponential sum
becomes 
\begin{align}
	\sum_{x\ppmod{p_1g_1}}e\Big(\frac{xn_2\overline{\tau_1\tau_2}}{p_1g_1}\Big)&\sum_{\substack{b_1\ppmod{p_1g_1}\\ (b_1, p_1)=1}}e\Big(\frac{(b_1h+\overline{b_1 }d_1)\overline{\tau_1\tau_2}}{p_1g_1}\Big)S(\bar{b}_1, x\overline{\tau_1\tau_2}^2n_1^2, p_1g_1)\\
	& \times \sum_{\substack{b_2\ppmod{p_1g_1}\\ (b_2, p_1)=1}}e\Big(\frac{-(b_2h+\overline{b_2 }d_2)\overline{\tau_1\tau_2}}{p_1g_1}\Big)S(\overline{b_2}, x\overline{\tau_1\tau_2}^2n_1^2, p_1g_1), 
\end{align}
which after summing over $x$ becomes in the same way as before
\begin{equation}\label{p1g1}
\begin{split}
	\sum_{\substack{y_1, y_2\ppmod{p_1g_1}\\ (y_1y_2, p_1)=1\\ \tau_1\tau_2n_2+n_1^2(y_1+y_2)\equiv 0\ppmod{p_1g_1}}}p_1g_1&\sum_{\substack{b_1 \ppmod {p_1g_1}\\ (b_1, p_1)=1}}e\Big(\frac{(b_1h+\overline{b_1}d_1)\overline{\tau_1\tau_2}+\overline{b_1y_1}}{p_1g_1}\Big)\\
	&\times \sum_{\substack{b_2 \ppmod {p_1g_1}\\ (b_2, p_1)=1}}e\Big(\frac{-(b_2h+\overline{b_2}d_2)\overline{\tau_1\tau_2}+\overline{b_2y_2}}{p_1g_1}\Big).
	\end{split}
\end{equation}
Applying  Weil's bound for the sums over $b_1, b_2$ and using that  $(hn_1, p_1) = 1$, we see that the above can be bounded by
\begin{align}
	p_1g_1 \sum_{\substack{y_1, y_2\ppmod{p_1g_1}\\ (y_1y_2, p_1)=1\\ \tau_1\tau_2n_2+n_1^2(y_1+y_2)\equiv 0\ppmod{p_1g_1}}}p_1p_2\ll  (p_1g_1)^3.
\end{align}
%5If $g_1=1$ and $p_1\nmid n_2$, then the above is of the form \eqref{weobtain} with ${\tt p}=p_1, {\tt p_1}={\tt p_2}=1, {\tt p'}=t_1t_2$ and can be bounded by $p_1^{5/2}$ since $p_1\nmid hn_2$. 
Modulo $\tau_1$, the exponential sum is of the form \eqref{t1exp} with ${\tt \tau_1'}=p_1g_1\tau_2, {\tt p_1}={\tt p_2}=1$ and the same proof gives the bound 
\begin{align}
	\tau_1^{5/2+\varepsilon}(hn_2, \tau_1)^{1/2}.
\end{align}
Modulo $\tau_2$, the exponential sum is of the form \eqref{t2exp} with ${\tt \tau_2'}=p_1g_1\tau_1, {\tt p_1}={\tt p_2}=1$ and so we obtain the bound $\tau_2^{3+\varepsilon}$. Combining all these estimates, we have a final bound 
\begin{align}
	(p_1g_1)^3\tau_1^{5/2+\varepsilon} (hn_2, \tau_1)^{1/2}\tau_2^{3+\varepsilon}%\ll p_1^{3}(g_1t_1t_2)^{5/2+\varepsilon}(hn_2, t_1)^{1/2}(p_1, g_{11})^{1/2}(g_{12}t_2)^{1/2}
\end{align}
%by writing $g_1=g_{11}g_{12}$ where $g_{12}$ is the squarefull part of $g_1$. This 
which gives \eqref{Tbound} for $\mathcal T$ in the case $p_1=p_2$ since $g_1^{1/2}(h n_2, \tau_1)^{1/2} \leq (hn_2p_1, t_1)^{1/2}$. (The bound \eqref{Tbound} can be improved in the $p_1$-aspect in many cases, but the above suffices for our purposes as long as $n_2 \not= 0$.) 

\medskip

%\begin{align}
%	p_1^{5/2}t_1^{5/2}(hn_2p_1, t_1)^{1/2}t_2^{3+\varepsilon}
%\end{align} and completes the case when $p_1\nmid  n_2$. 
%If $p_1\mid n_2$, we have the trivial bound
%\begin{align}
%	(p_1g_1)^{3}t_1^{5/2+\varepsilon}(hn_2,t_1)^{1/2}t_2^{3+\varepsilon}\ll (p_1g_1)^{5/2}(t_1t_2)^{5/2+\varepsilon}(hn_2, p_1t_1)^{1/2}(g_1t_2)^{1/2+\varepsilon}
%\end{align} if $g_1>p$ or $g_1=1$. 
%If $g_1=1$, then the trivial bound can be written as 
%\begin{align}
%	(p_1g_1)^3 t_1^{5/2+\varepsilon}(hn_2, t_1)^{1/2}t_2^{3+\varepsilon}\ll (p_1g_1)^{3/2}(n_2, p_1)t_1^{5/2+\varepsilon}(hn_2, t_1)^{1/2}t_2^{3+\varepsilon}.
%\end{align}

Now we improve \eqref{Tbound} if in addition $n_2=0$. We revisit the sum modulo $p_1g_1$ and note that \eqref{p1g1} becomes  
\begin{align}\label{p1g1n20}
	\sum_{\substack{y_1\ppmod{p_1g_1}\\ (y_1, p_1)=1}}p_1g_1 \sum_{\substack{b_1 \ppmod {p_1g_1}\\ (b_1, p_1)=1}}e\Big(\frac{(b_1h+\overline{b_1}d_1)\overline{\tau_1\tau_2}+\overline{b_1y_1}}{p_1g_1}\Big)\sum_{\substack{b_2 \ppmod {p_1g_1}\\ (b_2, p_1)=1}}e\Big(\frac{-(b_2h+\overline{b_2}d_2)\overline{\tau_1\tau_2}-\overline{b_2y_1}}{p_1g_1}\Big).
\end{align}
Then the sum over $y_1$ gives  $p_1g_1\mathbf{1}_{b_1\equiv b_2 \bmod {p_1g_1}}$ so that \eqref{p1g1n20} equals 
\begin{displaymath}
	(p_1g_1)^2\sum_{\substack{b_1\ppmod{p_1g_1}\\ (b_1, p_1)=1}}e\Big(\frac{\overline{b_1 \tau_ 1\tau_2}(d_1-d_2)}{p_1g_1}\Big)\ll (p_1g_1)^2 (d_1 - d_2, p_1g_1). %(p_1g_1)^{5/2}(d_1-d_2, p_1g_1)^{1/2}\ll p_1^{5/2}(d_1-d_2, p_1)^{1/2}g_1^{3+\varepsilon}.
\end{displaymath}
Modulo $\tau_1\tau_2$, we use the same bound as before getting a final bound 
\begin{align}
	(p_1g_1)^2 (d_1-d_2, p_1g_1) (\tau_1\tau_2)^{3+\varepsilon}\ll p_1^2 (d_1-d_2, p_1) (g_1\tau_1\tau_2)^{3+\varepsilon}.
\end{align}
This completes the proof.
\end{proof}

Since both Lemma \ref{char1} and \ref{charT} feature the squarefull part of the modulus we record for future reference the following bound. Let $f \in \Bbb{N}$ and write $f= f_1f_2$ with $f_1$ squarefree, $f_2$ squarefull, $(f_1, f_2) = 1$ and use the same notation for $d = d_1d_2$.  Then 
\begin{equation}\label{squarefull}
\begin{split}
& \sum_{\substack{c \leq C\\ c \text{ squarefull}}} c^{1/2} (c, f) = \sum_{d \mid f} d \sum_{\substack{d \mid c \leq C\\ c \text{ squarefull}}} c^{1/2} = 
  \sum_{d \mid f} d\sum_{\delta \mid d^{\infty}}\sum_{\substack{c \leq C/(d_1^2 d_2 \delta)\\ c \text{ squarefull}}} (\delta d_1^2d_2c)^{1/2} \\
  & \ll C  \sum_{d \mid f} d_2^{1/2}\sum_{\delta \mid d^{\infty}} \delta^{-1/2} \ll C  f_2^{1/2} f^{\varepsilon}  .
  \end{split} 
\end{equation}

\section{Proof of Theorem \ref{thm1}}

We start by observing that without loss of generality we can and do assume for the proof of Theorems \ref{thm1} and \ref{thm1a} that $\lambda_1, \lambda_2, h$ are pairwise coprime. Indeed, if not, then they all must have common divisor $d> 1$, otherwise the equation $\lambda_1 m - \lambda_2 n = h$ has no solution. We can divide the entire equation by $d$, which in effect amounts to replacing $x$ with $x/d$ in the weight functions $W$ and $W_0$.  Hence the proof in the case $\lambda_1, \lambda_2, h$ pairwise coprime implies a fortiori the case of a non-trivial common divisor.  

For the rest of the argument all implied constants may depend on a small $\varepsilon$ and a large $A$, where applicable,  without displaying this in the $\ll$-notation. 

\medskip

We recall that the main object of interest is
$$\sum_{\lambda_1 m- \lambda_2 n= h} A(n, 1) W\Big( \frac{|\lambda_2|n}{x}\Big) \tau(m) W_0\Big( \frac{|\lambda_1|m}{x}\Big).$$
In the notation of Lemma \ref{Jut} this equals $S_1 + S_2$ where 
\begin{equation}\label{s1}
\begin{split}
S_1 &= \sum_{n, m} A(n, 1) W\Big( \frac{|\lambda_2|n}{x}\Big) \tau(m) W_0\Big( \frac{|\lambda_1|m}{x}\Big) \int_0^1   \chi(\alpha) e((\lambda_1m-\lambda_2n-h)\alpha) d\alpha,\\
S_2 &= \sum_{n, m} A(n, 1) W\Big( \frac{|\lambda_2|n}{x}\Big) \tau(m) W_0\Big( \frac{|\lambda_1|m}{x}\Big) \int_0^1 \big(   1 - \chi(\alpha)\big)e((\lambda_1m-\lambda_2n-h)\alpha) d\alpha.
\end{split}
\end{equation}
Recall that the function $\chi$ depends on a choice of data 
$$Q,  \quad \delta, \quad \omega$$
(which determine $L$), 
and a function $\psi$ which we fix once and for all. To simplify the notation, we make the general assumptions
$$\| \omega \|_{\infty} \ll x^{\varepsilon}, \quad   \log Q, |\log \delta | \asymp \log x, \quad L = Q^{2 + o(1)}.$$

\subsection{Estimation of $S_2$.}
In this subsection we estimate $S_2$. The final bound will be \eqref{s2final} below.
We choose a parameter $C$ with
$$\log C \asymp \log x$$
and invoke Lemma \ref{Farey} (cf.\ also \eqref{MaVi}) along with \eqref{Four} to rewrite $S_2$ as
\begin{equation}\label{s2}
\begin{split}
\sum_{n, m} & A(n, 1) W\Big( \frac{|\lambda_2|n}{x}\Big) \tau(m) W_0\Big( \frac{|\lambda_1|m}{x}\Big)\\
&\times \sum_{c \leq C} \sum_{\substack{b \, (\text{mod } c)\\ (b, c) = 1}} \int_{-\frac{1}{c(c+c')}}^{\frac{1}{c(c+c'')}} \Big(   1 - \chi\Big(\frac{b}{c} + z\Big) \Big)e\Big((\lambda_1m-\lambda_2n-h)\Big(\frac{b}{c} + z\Big)\Big) dz.
\end{split}
\end{equation}
Finally we choose another parameter 
$$C_0 < C, \quad \log C_0 \asymp \log x$$
and split the previous sum into ``major arcs'' with $c \leq C_0$ and ``minor arcs'' $c \geq C_0$, which we call $S_{2, 0}$ and $S_{2, 1}$ respectively. 

\medskip

We estimate the the minor arc contribution as follows:
\begin{displaymath}
\begin{split}
|S_{2, 1}| \leq  & \Big( \sum_{C_0 < c \leq C} \sum_{\substack{b \, (\text{mod } c)\\ (b, c) = 1}}\int_{-\frac{1}{c(c+c')}}^{\frac{1}{c(c+c'')}}  \Big|   1 - \chi\Big(\frac{b}{c} + z\Big) \Big|^2 dz\Big)^{1/2} \\
&\times \Big( \sum_{C_0 < c \leq C} \sum_{\substack{b \, (\text{mod } c)\\ (b, c) = 1}}\int_{-\frac{1}{c(c+c')}}^{\frac{1}{c(c+c'')}} \Big| \sum_n A(n, 1) W\Big( \frac{|\lambda_2 |n}{x}\Big) e\Big( - \lambda_2n\Big(\frac{b}{c} + z\Big)\Big)  \Big|^2 dz\Big)^{1/2} \\
& \times  \max_{C_0 < c \leq C} \max_{\substack{b \, (\text{mod } c)\\ (b, c) = 1}} \max_{-\frac{1}{c(c+c')}<z<\frac{1}{c(c+c'')} } \Big|\sum_m  \tau(m) W_0\Big( \frac{|\lambda_1|m}{x}\Big)e\Big( \lambda_1m\Big(\frac{b}{c} + z\Big)\Big)\Big|.
\end{split}
\end{displaymath}
Since the intervals do not overlap and $c+c', c+c''>C$, we can replace the $b, c$-sum and $z$-integral in the first two lines by an integral over $[0, 1]$, take max over $z$ in the last line in a larger range $|z|\leq (cC)^{-1}$ and obtain
\begin{equation}\label{s21minor}
\begin{split}
S_{2, 1} \ll & \Big(   \int_{0}^1  |   1 - \chi(z)  |^2 dz\Big)^{1/2} \Big(   \sum_n \Big|A(n, 1) W\Big( \frac{|\lambda_2|n}{x}\Big)  \Big|^2  \Big)^{1/2} \\
&  \max_{C_0 < c \leq C} \max_{\substack{b \, (\text{mod } c)\\ (b, c) = 1}} \max_{|z|\leq (cC)^{-1}} \Big|\sum_m  \tau(m) W_0\Big( \frac{|\lambda_1|m}{x}\Big)e\Big( \lambda_1m\Big(\frac{b}{c} + z\Big)\Big)\Big|.
\end{split}
\end{equation}
Let us write
$$\frac{\lambda_1}{c} = \frac{\lambda_1'}{\cm}, \quad \frac{\lambda_2}{c} = \frac{\lambda_2'}{\cn}$$
where the right hand sides are in lowest terms. 
We apply Lemma \ref{Jut} for the first factor, the Rankin-Selberg bound  for the second factor and Lemmata \ref{vortau} and \ref{specialtau} with $\cm$ in place of $c$ and $X = x/|\lambda_1|$, $Z = z |\lambda_1|$ for the last factor, which gives the bound
\begin{equation}\label{msum}
\begin{split}
& \frac{x^{1+\varepsilon} }{|\lambda_1|\cm} \Big(1 +  \sum_{m} \tau(m) \Big[ \Big(1 + \frac{mx}{|\lambda_1|\cm^2}\Big)^{-A} + \frac{1}{1+x|z|} \Big(1 + \frac{mx}{|\lambda_1|\cm^2 (1 + x|z|)^2}\Big)^{-A} \Big] \Big)\\
& \ll \frac{x^{1+\varepsilon} }{|\lambda_1|\cm} \Big(1 + \frac{\cm^2(1 + x|z|) |\lambda_1|}{x} \Big) =x^{\varepsilon} \Big(\frac{x }{|\lambda_1|\cm}  + \cm(1 + x|z|)    \Big)
\end{split}
\end{equation}
for the $m$-sum.  In this way we obtain
\begin{equation}
\label{s21}
\begin{split}
S_{2, 1}&  \ll  x^\varepsilon \frac{Q}{L \delta^{1/2}} \Big(\frac{x}{|\lambda_2|}\Big)^{1/2}  \max_{C_0 < c \leq C}  \max_{|z|\leq (cC)^{-1}} \Big( \frac{x}{|\lambda_1| \cm} +  \cm (1 + x|z|)\Big) \\
& \ll x^{\varepsilon}  \frac{1}{Q\delta^{1/2}} \Big(\frac{x}{|\lambda_2|}\Big)^{1/2}  \Big( \frac{x}{C_0} +  C\Big).
\end{split}
\end{equation}

\medskip

We now return to \eqref{s2} and estimate  the major arc contribution $S_{2, 0}$ where $c \leq C_0$. For each $c\leq C_0$, we use \eqref{Fareycb} so that 
\begin{equation}
	\begin{split}
	S_{2,0}=&\sum_{c\leq C_0}\int_{-1/cC}^{1/cC}\frac{1}{c}\sum_{u\ppmod c}\sum_{t\in I(c,z)}e\Big(\frac{ut}{c}\Big) \sum_{\substack{b\ppmod c\\ (b,c)=1}}e\Big(\frac{u\bar{b}-hb}{c}\Big)e(-hz)\\
	&\quad\times \sum_{n}A(n,1)W\Big(\frac{|\lambda_2|n}{x}\Big)e(-\lambda_2nz)e\Big(\frac{-\lambda_2nb}{c}\Big)\\
	&\quad \times \sum_{m}\tau(m)e(\lambda_1mz)e\Big(\frac{\lambda_1mb}{c}\Big)W_0\Big(\frac{|\lambda_1|m}{x}\Big)\Big(1-\chi\Big( \big(\frac{b}{c}+z\big)\Big)\Big)dz.
	\end{split}
\end{equation}We start by dualizing the $m$-sum by Lemma \ref{vortau}. We also apply Lemma \ref{specialtau} and see that the dual sum is, up to a negligible error, restricted to
$$m \ll  x^{\varepsilon} \frac{\cm^2(1 + x|z|)^2}{x/|\lambda_1|} \ll x^{\varepsilon} |\lambda_1| \Big(\frac{C_0^2}{x} + \frac{x}{ C^2}\Big).$$
Let us assume
\begin{equation}\label{C0Ccondition}
|\lambda_1|^{1/2}	C_0 \leq x^{1/2 - \eta}, \quad C \geq   ( |\lambda_1|x)^{1/2 + \eta}
	\end{equation}
for some fixed $\eta > 0$. Then the dual sum is negligible. Moreover, a simple integration by parts argument shows that also the main term is negligible unless 
\begin{equation}\label{z}
   |z| \ll x^{\varepsilon-1}
\end{equation}   
for any $\varepsilon > 0$. 
    Finally we note that the main term is independent of $b$ and the $t$-sum can be bounded by $\frac{c}{1+|u|}$ for $|u| \leq c/2$.  Thus we obtain
\begin{displaymath}
\begin{split}
S_{2,0} \ll &   \sum_{c \leq C_0} \int_{|z| \leq x^{-1+\varepsilon}} \frac{x \log x}{|\lambda_1|\cm} \sum_{|u|\leq c/2}\frac{1}{1+|u|}\\
& \times\Big| \sum_{\substack{b \, (\text{mod } c)\\ (b, c) = 1}} e\Big(\frac{u\bar{b}-hb}{c}\Big) \sum_{n }   A(n, 1) W\Big( \frac{|\lambda_2|n}{x}\Big)e(-\lambda_2 nz)e\Big(\frac{-\lambda_2nb}{c}\Big)  \Big(   1 - \chi\Big(\frac{b}{c} + z\Big) \Big) \Big|dz .\end{split}
\end{displaymath}
We now dualize the $n$-sum using Lemma \ref{vorGL3} getting
\begin{displaymath}
\begin{split}
S_{2,0} \ll &   \sum_{c \leq C_0}\int_{|z| \leq x^{-1+\varepsilon}} \frac{x^2 \log x}{|\lambda_1\lambda_2|\cm \cn^2} \sum_{|u|\leq c/2}\frac{1}{1+|u|}\Big| \sum_{\substack{b \, (\text{mod } c)\\ (b, c) = 1}} e\Big(\frac{u\bar{b}-hb}{c}\Big) \sum_{\pm} \sum_{n_2} \sum_{n_1 \mid \cn} n_1 A(n_1, n_2)  \\
&   \times S\Big(-\overline{\lambda'_2 b}, \pm n_2, \frac{\cn}{n_1}\Big)\int_0^{\infty}W(y)e(-\sgn(\lambda_2)yzx)  V^{\pm}\Big(\frac{n_1^2n_2 xy}{\cn^3 |\lambda_2|}\Big) dy\, 
 \Big(   1 - \chi\Big(\frac{b}{c} + z\Big) \Big) \Big|dz .\end{split}
\end{displaymath}
A simple integration by parts argument using the bounds in Lemma \ref{vorGL3} (recall \eqref{z}) shows that
\begin{equation}\label{intbyparts}
\int_0^{\infty}W(y)e(-\sgn(\lambda_2)yzx)  V^{\pm}\Big(\frac{n_1^2n_2 xy}{|\lambda_2|\cn^3}\Big) dy \ll x^{ \varepsilon} \Big(1 + \frac{n_1^2n_2x}{|\lambda_2|\cn^3}\Big)^{-A} \Big(\frac{n_1^2n_2x}{|\lambda_2|\cn^3}\Big)^{-1/2}
\end{equation}
for any $\varepsilon, A > 0$.  We conclude
\begin{displaymath}
\begin{split}
S_{2,0} \ll &   \sum_{c \leq C_0} \int_{|z| \leq x^{-1+\varepsilon}} \frac{x^{3/2+\varepsilon}}{|\lambda_1||\lambda_2|^{1/2}\cm \cn^{1/2}} \sum_{\pm} \sum_{n_2}\frac{1}{n_2^{1/2}}\sum_{n_1 \mid \cn}| A(n_1, n_2)|\Big(1 + \frac{n_1^2n_2x}{|\lambda_2|\cn^3}\Big)^{-A} \\
&\quad\sum_{|u|\leq c/2}\frac{1}{1+|u|} \Big| \sum_{\substack{b \, (\text{mod } c)\\ (b, c) = 1}} e\Big(\frac{u\bar{b}-hb}{c}\Big)  S\Big(-\bar{b}, \pm \overline{\lambda'_2} n_2, \frac{\cn}{n_1}\Big) 
 \Big(   1 - \chi\Big(\frac{b}{c} + z\Big) \Big) \Big|dz .\end{split}
\end{displaymath}
We split this term into two parts according to the term $(1 - \chi(\frac{b}{c} + z))$ and call them $S_{2,0,0}$ and $S_{2,0,1}$.  For the contribution of the first summand  we insert the bound for Lemma \ref{char1} with $n_1(\lambda_2, c)$ in place of $n_1$ (and with the same notation $c = c_1 c_2$ where $c_2$ is the squarefull part of $c$) %getting bound
%$$c^{1+\varepsilon} c_2\Big( \frac{(\lambda_2, c, u)}{(\lambda_2, c)}\Big)^{1/2}.$$
%Noting that $\tilde{c} (c, \lambda_=  
and   obtain (recall \eqref{hecke} and \eqref{squarefull})
\begin{equation}\label{s200}
\begin{split}
S_{2,0,0} & \ll  \sum_{c \leq C_0}   \frac{x^{1/2+\varepsilon} c c_2^{1/2}}{|\lambda_1||\lambda_2|^{1/2}\cm \cn^{1/2}} \sum_{n_2} \frac{1}{n_2^{1/2}}  \sum_{n_1 \mid \cn}| A(n_1, n_2)|\Big(1 + \frac{n_1^2n_2x}{|\lambda_2|\cn^3}\Big)^{-A} \\
%\end{split}
%\end{displaymath}
%Using the bound
%$$\sum_{n_2 \leq X} |A(n_1, n_2)| \ll X^{1+\varepsilon} n_1^{1/2}$$
%which follows from Rankin--Selberg and trivial bounds towards the Ramanujan conjecture, 
%We conclude
%\begin{equation}
%\begin{split}\label{s200}
%S_{2,0,0} 
& \ll \sum_{c \leq C_0}   \frac{x^{1/2+\varepsilon}  c c_2^{1/2}}{|\lambda_1| |\lambda_2|^{1/2}\cm \cn^{1/2}}    \sum_{n_1 \mid \cm}n_1^{1/2} \frac{|\lambda_2|^{1/2} \cn^{3/2}}{x^{1/2} n_1} \ll x^{\varepsilon} \sum_{c \leq C_0} \frac{cc_2^{1/2} (c, \lambda_1)}{|\lambda_1|} \ll x^{\varepsilon}  \frac{C_0^2}{|\lambda_1|^{1/2}}. 
\end{split}
\end{equation}
(Here we could tigthen the estimate slightly if $\lambda_1$ is assumed to be squarefree or close to squarefree.)
 
\medskip

To deal with $S_{2, 0, 1}$, we insert \eqref{repchi} getting
\begin{equation}\label{s201}
\begin{split}
& S_{2,0, 1}\ll    \sum_{c \leq C_0} \int_{|z| \leq x^{-1+\varepsilon}} \frac{x^{3/2+\varepsilon}}{|\lambda_1||\lambda_2|^{1/2}\cm \cn^{1/2}} \sum_{\pm} \sum_{n_2}\sum_{n_1 \mid \cn} \frac{| A(n_1, n_2)|}{n_2^{1/2}} \Big(1 + \frac{n_1^2n_2x}{|\lambda_2|\cn^3}\Big)^{-A} \sum_{|u|\leq c/2}\frac{1}{1+|u|}\\
&\times \Big| \sum_{\substack{b \, (\text{mod } c)\\ (b, c) = 1}} e\Big(\frac{u\bar{b}-hb}{c}\Big)  S\Big(-\bar{b}, \pm \overline{\lambda'_2}n_2, \frac{\cn}{n_1}\Big) 
\frac{1}{\delta L} \sum_q \omega(q) \sum_{d\mid q} \mu\Big(\frac{q}{d}\Big) \sum_{\ell \equiv b d\, (\text{mod } c)}\psi\Big( \frac{1}{\delta} \Big( \frac{\ell}{cd} + z\Big)\Big) \Big|dz .\end{split}
\end{equation}
We first treat the contribution $\ell = 0$. In this case we must have $c\mid d$, and Lemma \ref{char1} (recall the notation $c=c_1c_2$ where $c_2$ is the squarefull part of $c$)  implies the bound (cf. \eqref{s200})
\begin{equation}\label{l=0}
	\ll x^{\varepsilon} \sum_{c \leq C_0} \frac{c c_2^{1/2}}{|\lambda_1|} (c, \lambda_1)\frac{1}{\delta L} \sum_{c\mid  q } \omega(q)\tau(q) \ll x^{\varepsilon} \sum_{c \leq C_0}   \frac{c_2^{1/2} (c, \lambda_1)}{|\lambda_1|}\frac{Q}{\delta L} \ll x^{\varepsilon} \frac{ C_0}{|\lambda_1|^{1/2}\delta Q}.
\end{equation}

From now on we assume $\ell \not= 0$. Since $(b, c) = 1$, we must have $(\ell, c) = (d, c) = g_1$, say. The second line in \eqref{s201} equals
%\begin{displaymath}
%\begin{split}
%&\frac{1}{\delta L} \sum_{g_1g_2 = c} \sum_{\substack{(\ell, g_2) = 1\\ \ell \not= 0}}  \sum_{ (d, g_2) = 1} \sum_r \omega(dg_1r)    \mu(r)\psi\Big( \frac{1}{\delta} \Big( \frac{\ell}{cd} + z\Big)\Big)\\
%&\times  \sum_{\substack{b \, (\text{mod } c)\\ (b, c) = 1\\ b \equiv \bar{d} \ell \, (\text{mod } g_2)}} e\Big(\frac{u\bar{b}-hb}{c}\Big)  S\Big(-\bar{b}, \pm \overline{\lambda_2'} n_2, \frac{\cn}{n_1}\Big) %+ O\Big(x^{\varepsilon} \frac{ C_0}{\delta Q}\Big). 
%\end{split}
%\end{displaymath}
%Moving the $b$-sum outside, the main term becomes
\begin{equation}\label{line2}
\begin{split}
\frac{1}{\delta L}&\sum_{\substack{b\ppmod c\\ (b,c)=1}}e\Big(\frac{u\bar{b}-hb}{c}\Big)S\Big(-\bar{b}, \pm \overline{\lambda'_2}n_2, \frac{\cn}{n_1}\Big)\\
&\times \sum_{g_1g_2=c}\sum_{\substack{(\ell, g_2)=1\\ \ell\not=0}}\sum_{r}\mu(r)\sum_{\substack{(d, g_2)=1\\ d\equiv \bar{b}\ell \ppmod {g_2}}}\omega(dg_1r)\psi\Big( \frac{1}{\delta} \Big( \frac{\ell}{cd} + z\Big)\Big). 
\end{split}
\end{equation}
%For notational simplicity let us denote the $b$-sum as $\mathcal{S}_{h, n_2, n_1}(d, \ell; g_2, c)$

It is at this point that we choose the function $\omega$. To this end we write $Q = Q_1Q_2$ with two parameters $1 \leq Q_1, Q_2 \leq Q$ and write
\begin{equation}\label{Qweight}
\omega(q) = \sum_{\substack{\frac{1}{2}Q_1 \leq p \leq Q_1\\ p \text{ prime}\\ p \nmid h\lambda_1 \lambda_2}} \sum_{\substack{t\in \Bbb{N}\\ pt = q}} \rho\Big(\frac{t}{Q_2}\Big)
\end{equation}
where $\rho : [1/2, 1] \rightarrow [0, 1]$ is a fixed smooth nonzero function. Analyzing the condition $pt = dg_1r$ yields three terms corresponding to $p \mid r$; $p\nmid r, p \mid g_1$; and $p \nmid rg_1$, $p \mid d$ (hence $(p, c) = 1$). Changing variables, this gives
\begin{align}\label{line}
&\sum_{g_1g_2=c}\sum_{\substack{(\ell, g_2)=1\\ \ell\not=0}}\sum_{r}\mu(r)\sum_{\substack{(d,g_2)=1\\ d\equiv \bar{b}\ell \ppmod{g_2}}}\omega(dg_1r)\psi\Big( \frac{1}{\delta} \Big( \frac{\ell}{cd} + z\Big)\Big)\\
=&-\sum_{\substack{\frac{1}{2}Q_1 \leq p \leq Q_1\\ p \text{ prime}\\ p \nmid h\lambda_1 \lambda_2}}  \sum_{g_1g_2 = c} \sum_{\substack{\ell \not= 0\\ (\ell, g_2) = 1}}   \sum_{(p, r) = 1}   \mu(r)\sum_{ \substack{(d,g_2)=1\\ d\equiv \bar{b}\ell \ppmod {g_2}}}\rho\Big( \frac{dg_1r}{ Q_2}\Big)   \psi\Big( \frac{1}{\delta} \Big( \frac{\ell}{cd} + z\Big)\Big)\\
&+ \sum_{\substack{\frac{1}{2}Q_1 \leq p \leq Q_1\\ p \text{ prime}\\ p \nmid h\lambda_1 \lambda_2}}  \sum_{pg_1g_2 = c} \sum_{\substack{\ell \not= 0\\ (\ell, g_2) = 1}}  \sum_{(p, r) = 1}\mu(r)\sum_{\substack{(d,g_2)=1\\ d\equiv \bar{b}\ell \ppmod{g_2}}}  \rho\Big( \frac{dg_1r}{ Q_2}\Big)   \psi\Big( \frac{1}{\delta} \Big( \frac{\ell}{cd} + z\Big)\Big)\\
&+\sum_{\substack{\frac{1}{2}Q_1 \leq p \leq Q_1\\ p \text{ prime}\\ p \nmid  ch\lambda_1 \lambda_2}}  \sum_{g_1g_2 = c} \sum_{\substack{\ell \not= 0\\ (\ell, g_2) = 1}}   \sum_{(p, r) = 1}  \mu(r)\sum_{\substack{(d,g_2)=1\\ d\equiv \overline{p b}\ell \ppmod{g_2}}}\rho\Big( \frac{dg_1r}{ Q_2}\Big)   \psi\Big( \frac{1}{\delta} \Big( \frac{\ell}{cdp} + z\Big)\Big).
\end{align}
%\begin{displaymath}
%\begin{split}
%& -\frac{1}{\delta L}\sum_{\substack{\frac{1}{2}Q_1 \leq p \leq Q_1\\ p \text{ prime}}}  \sum_{g_1g_2 = c} \sum_{\substack{\ell \not= 0\\ (\ell, g_2) = 1}}  \sum_{ (d, g_2) = 1} \sum_{(p, r) = 1}  \rho\Big( \frac{dg_1r}{ Q_2}\Big)    \mu(r)\psi\Big( \frac{1}{\delta} \Big( \frac{\ell}{cd} + z\Big)\Big) \mathcal{S}_{h, n_2, n_1}(d, \ell; g_2, c)\\
%&+ \frac{1}{\delta L}\sum_{\substack{\frac{1}{2}Q_1 \leq p \leq Q_1\\ p \text{ prime}}}  \sum_{pg_1g_2 = c} \sum_{\substack{\ell \not= 0\\ (\ell, g_2) = 1}}  \sum_{ (d, g_2) = 1} \sum_{(p, r) = 1}  \rho\Big( \frac{dg_1r}{ Q_2}\Big)   \mu(r)\psi\Big( \frac{1}{\delta} \Big( \frac{\ell}{cd} + z\Big)\Big) \mathcal{S}_{h, n_2, n_1}(d, \ell; g_2, c)\\
%&+ \frac{1}{\delta L}\sum_{\substack{\frac{1}{2}Q_1 \leq p \leq Q_1\\ p \text{ prime}\\ (p, c) = 1}}  \sum_{g_1g_2 = c} \sum_{\substack{\ell \not= 0\\ (\ell, g_2) = 1}}  \sum_{ (d, g_2) = 1} \sum_{(p, r) = 1}  \rho\Big( \frac{dg_1r}{ Q_2}\Big)   \mu(r)\psi\Big( \frac{1}{\delta} \Big( \frac{\ell}{cdp} + z\Big)\Big) \mathcal{S}_{h, n_2, n_1}(pd, \ell; g_2, c) + O\Big(x^{\varepsilon} \frac{ C_0}{\delta Q}\Big).
%\end{split}
%\end{displaymath}
We now prepare for the next important step, Poisson summation in $d$. This can only be done efficiently if we have no arithmetic condition in the sum over $t$ in \eqref{Qweight}, in particular we cannot restrict to $t$ prime as in \cite{Mu1}. 

We apply a smooth partition of unity and localize $d\asymp D$ for some parameter $1 \leq D \leq Q_2/g_1r$ with a smooth weight function $v(d/D)$. We make the general assumption
\begin{equation}\label{delta}
\delta \gg x^{-1 + \varepsilon}
\end{equation}
so that $z/\delta \ll 1$. We remember the size condition $\ell \ll c D \varpi \delta$ with $\varpi \in \{1, p\}$, depending on the summand. For $\varpi \in \{1, p\}$ we have
\begin{equation}
	\begin{split}
%		&\sum_{(d, g_2) = 1} v\Big(\frac{d}{D}\Big)  \rho\Big( \frac{dg_1r}{ Q_2}\Big)   \psi\Big( \frac{1}{\delta} \Big( \frac{\ell}{cd\varpi} + z\Big)\Big) \mathcal{S}_{h, n_2, n_1}(\varpi d, \ell; g_2, c)\\
%		&=\sum_{\substack{b\ppmod c\\ (b,c)=1}}e\Big(\frac{-hb}{c}\Big)S(-\bar{b}, \pm n_2, \frac{c}{n_1})
		\sum_{\substack{(d,g_2)=1\\ d\equiv \overline{\varpi b}\ell\ppmod{g_2}}}  v\Big(\frac{d}{D}\Big)  \rho\Big( \frac{dg_1r}{ Q_2}\Big)   \psi\Big( \frac{1}{\delta} \Big( \frac{\ell}{cd\varpi} + z\Big)\Big)
	\end{split}
\end{equation}
After Poisson summation, this becomes
\begin{align}\label{Poissond}
\frac{1}{g_2}e\Big(\frac{\overline{\varpi b}\ell d}{g_2}\Big)\int_\RR v\Big(\frac{\xi}{D}\Big)\rho\Big(\frac{\xi g_1r}{Q_2}\Big)\psi \Big(\frac{1}{\delta}\Big( \frac{\ell}{c\xi\varpi}+z\Big)\Big)e\Big(\frac{-\xi d}{g_2}\Big)d\xi.
\end{align}
Integration by parts shows that the integral is $\ll_A D(1 + dD/g_2)^{-A}$ for every $A > 0$.
The character sum over $b$ becomes 
\begin{align}
&\sum_{\substack{b\ppmod c\\ (b,c)=1}}e\Big(\frac{u\bar{b}-hb}{c}\Big)S\Big(-\bar{b}, \pm \overline{\lambda_2'}n_2, \frac{\cn}{n_1}\Big)e\Big(\frac{\overline{\varpi b}\ell d}{g_2}\Big)\\
&=\sum_{\substack{b\ppmod c\\ (b,c)=1}}e\Big(\frac{-hb+\bar{b}\bar{\varpi}(\ell dg_1+\varpi u)}{c}\Big)S\Big(-\bar{b}, \pm\overline{\lambda_2'} n_2, \frac{c}{n_1(c, \lambda_2)}\Big)\ll c^{1+\varepsilon}c_{2}^{1/2}
\end{align}
by Lemma \ref{char1} (with the usual notation that $c_2$ denotes the squarefull part of $c$), so that we can conclude that the contribution from third summand in \eqref{line} (which is the hardest) to \eqref{line2} is bounded by 
\begin{displaymath}
	\begin{split}
		& \ll \frac{1 }{\delta L}  Q_1  \sum_{g_1g_2 = c} \sum_{r \leq Q_2}   \max_{D \ll \frac{Q_2}{g_1 r}}   (c D Q_1 \delta) \frac{1}{g_2} c^{1+\varepsilon} c_2^{1/2}  D \Big(1 + \frac{g_2}{D}\Big)\\
		&  \ll   \frac{x^{\varepsilon} }{  Q^2}  Q_1  \sum_{g_1g_2 = c} \sum_{r \leq Q_2}  \Big( \frac{Q_2}{g_1 r} Q_1  \Big)   c^{2} c_2^{1/2}   \Big( \frac{Q_2}{g_1g_2 r}  + 1 \Big) \\ 
		&  \ll   \frac{x^{\varepsilon} }{  Q}  Q_1    c^{2} c_2^{1/2}   \Big( \frac{Q_2}{c }  + 1 \Big) \ll  x^{\varepsilon} c_2^{1/2}  \Big(c+ \frac{c^2}{Q_2}\Big).  \end{split}
\end{displaymath}
The other two summands in \eqref{line}  are dominated by this quantity. Substituting back into \eqref{s201}, together with \eqref{hecke}, \eqref{squarefull} and \eqref{l=0}, we obtain the total bound
\begin{equation}\label{s201bound}
	S_{2, 0, 1} \ll x^{\varepsilon} \Big(\sum_{c \leq C_0} \frac{c_2^{1/2} (c, \lambda_1)}{|\lambda_1|} \Big(c+ \frac{c^2}{Q_2}\Big) +   \frac{ C_0}{|\lambda_1|^{1/2}\delta Q}\Big) \ll \frac{x^{\varepsilon}}{|\lambda_1|^{1/2}}\Big(C_0^2 + \frac{C_0^3}{Q_2} + \frac{C_0}{\delta Q}\Big).
	\end{equation}

Combining this with \eqref{s21} and \eqref{s200} together with $\delta\gg x^{-1+\varepsilon}$, we arrive at the final bound
\begin{equation}\label{s2final}
	\begin{split}S_2 %&\ll x^{\varepsilon} \Big(\frac{1}{Q\delta^{1/2}} \Big(\frac{x}{|\lambda_2|}\Big)^{1/2}   \Big( \frac{x}{C_0} +  C\Big) + \frac{1}{|\lambda_1|}\Big(C_0^2 + \frac{C_0^3}{Q_2} + \frac{C_0}{\delta Q}\Big)\Big)\\
		& \ll x^{1+\varepsilon }\Big(\frac{x}{|\lambda_2|^{1/2}QC_0}+\frac{C}{|\lambda_2|^{1/2}Q}+\frac{C_0^2}{|\lambda_1|^{1/2}x}+\frac{C_0^3}{|\lambda_1|^{1/2}xQ_2}+\frac{C_0}{|\lambda_1|^{1/2}Q}\Big).
	\end{split}
\end{equation}

\subsection{Estimation of $S_1$}

We now estimate $S_1$, defined in \eqref{s1}.  The final bound is \eqref{s1final} below. The first steps follow  Munshi \cite[Section 4]{Mu1} with a different choice of $\mathcal Q$, but at some point we need to diverge from his analysis. 

 Using the definition given in Lemma \ref{Jut}, we see that 
\begin{displaymath}
	\begin{split}
	S_1= \frac{1}{\delta L}\sum_{\substack{q\in \mathbb N\\ q\leq Q}}\omega(q)&\sum_{\substack{a\ppmod {q}\\ (a,q)=1}}\sum_{m,n}A(n,1)W\Big(\frac{|\lambda_2|n}{x}\Big)\tau(m)W_0\Big(\frac{|\lambda_1|m}{x}\Big)\\
	&\times \int_{\Bbb{R}}e\Big(\Big(\frac{a}{q}+z\Big)(\lambda_1m-\lambda_2n-h)\Big)\psi \Big(\frac{z}{\delta}\Big)dz. 
	\end{split}
\end{displaymath}
As before we write
$$\frac{\lambda_1}{q} = \frac{\lambda_1'}{\qm}, \quad \frac{\lambda_2}{q} = \frac{\lambda_2'}{\qn}$$
in lowest terms and recall that in the decomposition $q = pt$ of \eqref{Qweight} we have $p \nmid h\lambda_1\lambda_2$. 

We now apply Lemma \ref{vortau} and Lemma \ref{vorGL3} so that 
%\begin{align}
	%$S_1=S_{1,0}+S_{1,1}$ 
%\end{align}
%where 
\begin{displaymath}
\begin{split}
S_1 = &\frac{1}{\delta L}\sum_{q\leq Q}\sum_{\substack{a\ppmod q\\ (a,q)=1}} \int \omega(q)\frac{x^2}{|\lambda_1\lambda_2|\qm \qn^2} \Bigg( \int_0^\infty W_0(\xi)e(\sgn(\lambda_1)z\xi x)\Big(\log \frac{\xi x}{\qm^2}+2\gamma\Big)d\xi \\
&\quad\quad \quad+ \sum_{\pm}\sum_{m}\tau(m)e\Big(\pm\frac{\overline{\lambda_1'a}m}{\qm}\Big) \int_0^\infty W_0(\xi)e(\sgn(\lambda_1)z\xi x)J^\pm\Big(\frac{\sqrt{m\xi x}}{|\lambda_1|^{1/2}\qm}\Big)d\xi \Bigg)\\
&  \times\sum_{\pm}\sum_{n_2}\sum_{n_1\mid \qn}n_1A(n_1, n_2) S\Big(-\overline{\lambda_2'a},\pm n_2, \frac{\qn}{n_1}\Big)\\
& \times \int_{0}^\infty W(y)e(-\sgn(\lambda_2)zyx)V^\pm \Big(\frac{n_1^2n_2yx}{|\lambda_2|\qn^3}\Big)dy\, \,e\Big(\frac{-ah}{q}\Big)e(-zh)\psi\Big(\frac{z}{\delta}\Big)dz. 
\end{split}
\end{displaymath}
%\begin{equation}
%	\begin{split}
%	 S_{1,0}=&\frac{1}{\delta L}\sum_{q\leq Q}\sum_{\substack{a\ppmod q\\ (a,q)=1}}\int \omega(q)\frac{x^2}{q^3}\int_0^\infty W_0(\xi)e(z\xi x)\Big(\log \frac{\xi x}{q^2}+2\gamma\Big)d\xi \,  \sum_{\pm}\sum_{n_2}\sum_{n_1\mid q}n_1A(n_1, n_2)\\
%	 &   \times S\Big(-\bar{a},\pm n_2, \frac{q}{n_1}\Big)\int_{0}^\infty W(y)e(-zyx)V^\pm \Big(\frac{n_1^2n_2yx}{q^3}\Big)dy\, e\Big(\frac{-ah}{q}\Big)e(-zh)\psi\Big(\frac{z}{\delta}\Big)dz,\\
%	  S_{1,1}=&\frac{1}{\delta L}\sum_{q\leq Q}\sum_{\substack{a\ppmod q\\ (a,q)=1}}\int\omega(q)\frac{x^2}{q^3}\sum_{\pm}\sum_{n}\tau(n)e\Big(\pm\frac{\bar{a}n}{q}\Big)\int_0^\infty W_0(\xi)e(z\xi x)J^\pm\Big(\frac{\sqrt{n\xi x}}{q}\Big)d\xi \\
%	 &  \times  \sum_{\pm}\sum_{n_2}\sum_{n_1\mid q}n_1A(n_1, n_2)S\Big(-\bar{a},\pm n_2, \frac{q}{n_1}\Big)\int_{0}^\infty W(y)e(-zyx)V^\pm \Big(\frac{n_1^2n_2yx}{q^3}\Big)dy\, e\Big(\frac{-ah}{q}\Big)e(-zh)\psi\Big(\frac{z}{\delta}\Big)dz.
%	 \end{split}
%\end{equation}
We now make the final choice
\begin{equation}\label{defdelta}
\delta = x^{-1+\varepsilon}
\end{equation}
so that the exponentials $e(\sgn(\lambda_1)z \xi x)$ and  $e(-\sgn(\lambda_2)zyx)$ are almost flat, but \eqref{delta} is satisfied. In particular, by Lemma   \ref{vorGL3} (cf.\ \eqref{intbyparts}), the $y$-integrals  is bounded by
$$\ll x^{\varepsilon} \frac{\qn^{3/2}|\lambda_2|^{1/2}}{x^{1/2} n_1 n_2^{1/2}} \Big(1 + \frac{n_1^2n_2 x}{\qn^3|\lambda_2|}\Big)^{-A}$$
for any $\varepsilon, A > 0$ (where $\varepsilon$ in \eqref{defdelta} and hence in \eqref{delta}, which is the same $\varepsilon$ as in \eqref{z}, has to be chosen accordingly in terms of the present $\varepsilon$ and $A$). 

%Using the bounds in Lemma \ref{specialtau} and Lemma \ref{vorGL3} and the notation of Lemma  % and Lemma 
%\ref{char1}, we have  
%\begin{equation}
%	\begin{split}
%				S_{1,0}&\ll \frac{1}{ L}  \sum_{q\leq Q}\omega(q)\frac{x^{3/2+\varepsilon}}{q^{3/2}} \sum_{\pm}\sum_{n_2}\frac{1}{n_2^{1/2}}\sum_{n_1\mid q}|A(n_1, n_2)|\Big(1+\frac{n_1^2n_2x}{q^3}\Big)^{-A} |\Sigma_{h, 1, 0, n_1, \pm n_2}(q)|\\
%		S_{1,1}&\ll \frac{1}{  L}  \sum_{q\leq Q}\omega(q)\frac{x^{3/2+\varepsilon}}{q^{3/2}} \sum_{\pm}\sum_{n_2}\frac{1}{n_2^{1/2}}\sum_{n_1\mid q}|A(n_1, n_2)|\Big(1+\frac{n_1^2n_2x}{q^3}\Big)^{-A}\\
%		&\quad \quad \quad \sum_{m}\tau(m)\Big(1+\frac{mx}{q^2}\Big)^{-A} \Big|\sum_{\substack{a\ppmod q\\ (a, q)=1}}e\Big(\frac{-ha+\bar{a}m}{q}\Big)S\Big(-\bar{a}, \pm n_2, \frac{q}{n_1}\Big)\Big|.  
%S_{1}&\ll \frac{1}{  L}  \sum_{q\leq Q}\omega(q)\frac{x^{3/2+\varepsilon}}{q^{3/2}} \sum_{\pm} \sum_{n_2}\frac{1}{n_2^{1/2}}\sum_{n_1\mid q}|A(n_1, n_2)|\Big(1+\frac{n_1^2n_2x}{q^3}\Big)^{-A}\\
%		&\quad \quad \quad 
%		\sum_{m \in \Bbb{Z}} \Big(1+\frac{|m|x}{q^2}\Big)^{-A} \big| \Sigma_{h, 1, -m, n_1, \pm n_2}(q)\big|. %\Big|\sum_{\substack{a\ppmod q\\ (a, q)=1}}e\Big(\frac{-ha+\bar{a}m}{q}\Big)S\Big(-\bar{a}, \pm n_2, \frac{q}{n_1}\Big)\Big|. 
%	\end{split}
%\end{equation}
We put the variables $n_1, n_2$ in dyadic ranges $H \leq n_1\leq  2H$, $N \leq n_2 \leq 2 N$ and denote by $S_1(N, H)$ the corresponding  contribution to $S_1$. In particular, we may at the cost of a negligible error assume that
%\begin{equation}\label{sizes}
  $ H^2 N  \ll Q^3 |\lambda_2| x^{\varepsilon - 1}.$ 
%\end{equation}
%The typical case is that $n_2$ has the maximal length $N \approx q^3/x$, while $H \approx 1$, but we will first provide a trivial bound that is useful if $N$ is small or $H$ is big.  %To this end, we decompose $q = q_1q_2$ uniquely with $q_1$ squarefree, $q_2$ squarefull and $(q_1, q_2) = 1$. 
 Using in addition  the bounds in Lemma \ref{specialtau} %tand  Lemma \ref{vorGL3} (cf.\ \eqref{intbyparts}) % and Lemma \ref{char1}, % together with Hecke relation and Rankin-Selberg theory, 
we see that 
\begin{displaymath}
\begin{split}
S_1(N, H)   \ll & \frac{x^{3/2+\varepsilon}}{|\lambda_1| |\lambda_2|^{1/2} \delta  LQ^{3/2}} \int_{z \ll \delta} \Big|\sum_{\substack{q=pt\asymp Q\\ p\asymp Q_1, t\asymp Q_2\\p\nmid h\lambda_1\lambda_2}} (t, \lambda_1)(t, \lambda_2)^{1/2}\sum_{\substack{n_1\asymp H\\ n_1\mid \qn}}\sum_{n_2\asymp N}\frac{|A(n_1, n_2)|}{\sqrt{n_2}}\\
&\times \sum_{m\in \Bbb{Z}} \Sigma_{h,m\overline{\lambda_1'} (t, \lambda_1), n_1 (t, \lambda_2), \pm n_2\overline{\lambda_2'}}(q)\Omega_{z, n_1, n_2}(m, q)\Big|\, dz
\end{split}
\end{displaymath} 
for any $\varepsilon > 0$ where 
\begin{equation}\label{Omega}
\Omega_{z, n_1, n_2}(m, q) \ll \tau(m)\Big(1 + \frac{|m|x}{|\lambda_1|\qm^2}\Big)^{-A} \Big(1 + \frac{n_1^2n_2 x}{|\lambda_2|\qn^3}\Big)^{-A}
\end{equation}
for any $A > 0$. %While most sums are estimated trivially, it is important to keep the $m, p$-sums both inside the absolute values (unlike the treatments in \cite{Mu1, Xi} for instance). 

%So far it was not necessary, but we observe that we may assume up front that $(\lambda_1, \lambda_2) = 1$, otherwise we can divide a common gcd from the equation $\lambda_1m - \lambda_2 = h$ (or the condition is void). This allows us to simplify $(t, \lambda_1)(t, \lambda_2)^{1/2} \leq (t, \lambda_1\lambda_2)$. 

The typical case is that $n_2$ has the maximal length $N \approx q^3/x$, while $H \approx 1$, but we will first provide a trivial bound that is useful if  $H$ is big.  To this end, we estimate all sums trivially and decompose $q = q_1q_2$ uniquely with $q_1$ squarefree, $q_2$ squarefull and $(q_1, q_2) = 1$. Using Lemma \ref{char1} and \eqref{Omega}, we see that %$$S_{1, 0 }(N, H)  + S_{1, 1}(N, H)$ is bounded by 
%so that the contribution from $n_1\sim H, n_2\sim N$ to $S_1$ can be bounded by
\begin{displaymath}
\begin{split}
%& \ll    \frac{1}{  L}  \sum_{q\leq Q} \frac{x^{3/2+\varepsilon}}{q^{3/2}} \sum_{\pm} \sum_{n_2 \asymp N}\frac{|A(n_1, n_2)| }{\sqrt{n_2}}\sum_{\substack{n_1 \asymp H\\ n_1\mid q}}
%		\sum_{m \in \Bbb{Z}}  \Big(1+\frac{|m|x}{q^2}\Big)^{-A} \big| \Sigma_{h, 1, -m, n_1, \pm n_2}(q)\big|\\
&S_1(N, H)  \ll	  \frac{x^{3/2+\varepsilon}}{|\lambda_1||\lambda_2|^{1/2}Q^{7/2}}\sum_{\substack{q =pt\asymp Q\\ p\asymp Q_1, t\asymp Q_2}} 
	%\substack{q=q_1q_2\ll Q\\ \mu^2(q_1)=1\\ q_2 \text{ squarefull}}} 
	(t, \lambda_1)(t, \lambda_2)^{1/2} \sum_{\substack{n_1\asymp H\\ n_1\mid \qn}}\sum_{n_2\asymp N}\frac{|A(n_1, n_2)|}{\sqrt{n_2}} \frac{q \sqrt{q_2}}{\sqrt{n_1(t, \lambda_2)}}  \\
	 & \quad \quad \quad\quad\quad\quad \times   \sum_{m \in \Bbb{Z}} \Big(1+\frac{|m|x}{|\lambda_1|\qm^2} + \frac{H^2Nx}{|\lambda_2|\qn^3}\Big)^{-A} \sqrt{(q_1, n_1(t, \lambda_2), m(t, \lambda_1))}\\
	&   \ll   \frac{x^{3/2+\varepsilon}}{|\lambda_1||\lambda_2|^{1/2}Q^2}   \sum_{ n_1\asymp H}\sum_{n_2\asymp N}\frac{|A(n_1, n_2)|}{\sqrt{n_1n_2}} \sum_{\substack{q \asymp Q \\ n_1 \mid \qn} } \frac{(q, \lambda_1)}{\sqrt{q_1}} \Big( \sqrt{(q, \lambda_2)(q_1, n_1)} +  \frac{|\lambda_1|\qm^2\sqrt{(q_1, \lambda_1)}}{x}\Big)\Big(1+ \frac{H^2Nx}{|\lambda_2|\qn^3}\Big)^{-A}.
\end{split}
  \end{displaymath}
  Here the first term in the penultimate parenthesis corresponds to $m=0$. Let $g = (\lambda_2, q)$. For any $n\in \Bbb{N}$ we have 
  \begin{displaymath}
  \begin{split}
  &  \sum_{\substack{q\asymp Q\\ n\mid \qn}}\frac{(q, \lambda_1)\sqrt{(q, \lambda_2)(q_1, n)}}{q_1^{1/2}} \leq |\lambda_1|  g^{1/2} \sum_{\substack{\qn\asymp Q/g\\ n\mid \qn}}\frac{ \sqrt{(\qn_1, n)}}{\qn_1^{1/2}}  \leq |\lambda_1|  g^{1/2} \sum_{ r\asymp Q/gn }\frac{   (r, n^{\infty})^{1/2}}{r_1^{1/2}}\\
    & \leq|\lambda_1| g^{1/2}  \sum_{\substack{\nu \mid n^{\infty} \\ \nu \leq  Q}} \nu^{1/2}  \sum_{r \asymp Q/gn\nu} \frac{1}{r_1^{1/2}} \ll |\lambda_1| g^{1/2}  \sum_{\substack{\nu \mid n^{\infty}\\ \nu \leq Q}} \nu^{1/2} \Big( \frac{Q}{gn\nu}\Big)^{1/2} \ll Q^{1/2 + \varepsilon} \frac{|\lambda_1|}{n^{1/2}}
 \end{split}
 \end{displaymath} 
  %Since $(\lambda_1, \lambda_2)=1$, we note that (with the notation $n_2$ to denote the squarefull part of $n$)
%  \begin{align}
%  	&\sum_{\substack{q\asymp Q\\ n_1\mid \qn}}\frac{(q, \lambda_1)\sqrt{(q, \lambda_2)(q_1, n_1)}}{q_1^{1/2}}=\sum_{g_2\mid \lambda_2}\sqrt{g_{22}}\sum_{\substack{\qn \asymp Q/g_2\\ n_1\mid \qn}}\frac{(\qn, \lambda_1)\sqrt{(q_1, n_1)}}{\sqrt{\qn_1}}\\
  %	&\ll \sum_{g_2\mid \lambda_2}\sqrt{g_{22}}\sum_{\substack{\qn_1\qn_2\asymp Q/g_2\\ n_{11}\mid \qn_1\\ n_{12}\mid \qn_2}}\frac{(\qn_1, \lambda_1)\sqrt{n_{11}}}{\sqrt{\qn_1}}(\qn_2, \lambda_1)\\
%  	& \ll \sum_{g_2\mid \lambda_2}\sqrt{g_{22}}(n_{11}, \lambda_1)\sqrt{\frac{\lambda_{11}}{(\lambda_{11}, n_{11})}}(n_{12}, \lambda_1)\sqrt{\lambda_{12}}\frac{Q^{1/2}}{g_2\sqrt{n_1}}\\
%  	 & \ll (n_1, \lambda_1)\sqrt{\lambda_1}\frac{Q^{1/2}}{\sqrt{H}}\\
%& \ll \sum_{g_2\mid \lambda_2}\sqrt{g_{22}}\sum_{\substack{\qn \asymp Q/g_2\\ n_{11}\mid \qn_1\\n_{12}\mid \qn_2}}|\lambda_1| \frac{\sqrt{n_{11}}}{\sqrt{\qn_1}}
% \ll|\lambda_1| \sum_{g_2\mid \lambda_2}\sqrt{g_{22}} \frac{Q^{1/2}}{g_2\sqrt{n_1}}\ll |\lambda_1|Q^{1/2}n_1^{-1/2}
%  \end{align}
where $\qn_1$, $q_1$, $r_1$ denotes the respective squarefree part of $\qn, q, r$. Moreover, 
  \begin{align}
  	\sum_{\substack{q\asymp Q\\ n_1\mid \qn}}\frac{(q, \lambda_1)}{q_1^{1/2}}\frac{|\lambda_1|\qm^2\sqrt{(q_1, \lambda_1)}}{x}\ll|\lambda_1| \sum_{\substack{q\asymp Q\\ n_1\mid q}}\frac{q^2}{x q_1^{1/2}}\ll \frac{|\lambda_1|Q^{5/2}}{x\sqrt{n_1}}, 
  \end{align}
%Using the Hecke relation $A(n_1, n_2) = \sum_{d\mid(n_1, n_2)} \mu(d) A(n_1/d, 1)A(1, n_2/d)$ and Rankin-Selberg theory, we obtain
and so 
\begin{equation}\label{trivialbound}
\begin{split}
S_{1  }(N, H) 
& \ll  \frac{x^{3/2+\varepsilon}}{|\lambda_1||\lambda_2|^{1/2}Q^2}   \sum_{ n_1\asymp H}\sum_{n_2\asymp N}\frac{|A(n_1, n_2)|}{\sqrt{n_1n_2}} \Big(\frac{|\lambda_1|Q^{1/2}}{\sqrt{n_1}} + \frac{|\lambda_1|Q^{5/2}}{x n_1^{1/2}}\Big)\Big(1+ \frac{H^2Nx}{|\lambda_2|Q^3}\Big)^{-A}\\
  & \ll  \frac{x^{3/2+\varepsilon}}{|\lambda_1||\lambda_2|^{1/2}Q^2} \sqrt{HN}\Big(\frac{|\lambda_1|Q^{1/2}}{H^{1/2}} + \frac{|\lambda_1|Q^{5/2}}{x H^{1/2}}\Big) \Big(1+ \frac{H^2Nx}{|\lambda_2|Q^3}\Big)^{-A}\\
  &\ll  x^{\varepsilon}\Big(\frac{x}{H} + \frac{Q^2}{H}\Big)
\end{split}
\end{equation}
%	  \begin{align}	  
%	&\ll   x^{1/2+\varepsilon}\sum_{n_1\sim H}\sum_{n_2\sim N}\frac{1}{\sqrt{n_1n_2}}|A(n_1, n_2)|\sum_{\substack{q=q_1q_2\ll Q/H\\ \mu^2(q_1)=1\\ q_2 \text{ squarefull}}}\frac{1}{\sqrt{q_1}}\\
%	&\ll x^{1/2+\varepsilon} \sqrt{\frac{Q}{H}}\Big(\sum_{n_1\sim H}\sum_{n_2\sim N} |A(n_1, n_2)|^2\Big)^{1/2}\Big(\sum_{n_1\sim H}\sum_{n_2\sim N}\frac{1}{n_1n_2}\Big)^{1/2}\\
%	& \ll x^{1/2+\varepsilon}\sqrt{ \frac{Q}{H}}H^{1/2}N^{1/2}\ll x^\varepsilon (xQN)^{1/2}\label{n_2<N}
%\end{align}
%using Lemma \ref{char1} together with Hecke relation and Rankin-Selberg theory. 
%Since the contribution from $N\gg \frac{x^\varepsilon Q^3}{H^2x}$ is negligible, we see that the contribution from $n_1\gg H$ to $S_1$ can be bounded by 
%\begin{align}\label{n1>H}
%\frac{x^\varepsilon Q^2}{H}.
%\end{align}
%The bounds in \eqref{n_2<N} and \eqref{n1>H} will be sufficient if 
%\begin{align}
%	N<x/Q \text{ and }H> x/Q^2.
%\end{align}
for any $\varepsilon > 0$. 

We keep this in mind for future reference and continue with a more sophisticated argument. % that makes use of the specific structure of the function $\omega$. 
From now on let us assume 
%For large $N$ and small $H$, we shall make use of the structure of $q$ to obtain further cancellations. Using the definition of $\omega(q)$ in \eqref{Qweight}, and additionally we assume that 
\begin{equation}\label{q1h}
	Q_1 \geq 10 H
\end{equation}
so that $n_1\mid \qn$ with $n_1\asymp H$ and $p\asymp Q_1, p\nmid \lambda_2$ implies $n_1\mid \breve{t} := t/(t, \lambda_2)$. It then follows that  %In the following we treat the term $S_{1, 1}(N, H)$, which is the harder of the two. The corresponding analysis for $S_{1, 0}(N, H)$ will be an easier variation. 
\begin{displaymath}
	\begin{split}
		S_1(N, H)   \ll & \frac{x^{3/2+\varepsilon}}{|\lambda_1| |\lambda_2|^{1/2} \delta  LQ^{3/2}} \int_{z \ll \delta} \sum_{t\asymp Q_2} (t, \lambda_1)(t, \lambda_2)^{1/2}\sum_{\substack{n_1\asymp H\\ n_1\mid \breve{t}}}\sum_{n_2\asymp N}\frac{|A(n_1, n_2)|}{\sqrt{n_2}}\\
		&\times \Bigg|\sum_{m\in \Bbb{Z}}  \sum_{\substack{p\asymp Q_1\\ p \text{ prime} \\ p \nmid h\lambda_1\lambda_2}}\Sigma_{h,m\overline{\lambda_1'} (t, \lambda_1), n_1 (t, \lambda_2), \pm n_2\overline{\lambda_2'}}(pt)\Omega_{z, n_1, n_2}(m, pt)\Bigg|\, dz.
	\end{split}
\end{displaymath} While most sums are estimated trivially, it is important to keep the $m, p$-sums both inside the absolute values (unlike the treatments in \cite{Mu1, Xi} for instance). 
We apply the Cauchy--Schwarz inequality to bound $S_1(N, H)$ by
%Cauchy in $t,  n_1, n_2$, we see that \eqref{prepoisson} is bounded by 
\begin{equation}
	\begin{split}
	&\frac{x^{3/2+\varepsilon}}{|\lambda_1||\lambda_2|^{1/2}\delta LQ^{3/2}} \int_{z \ll \delta} \Big(\sum_{\substack{n_1\asymp H}}\sum_{n_2\asymp N}\frac{|A(n_1, n_2)|^2}{n_2}\sum_{\substack{  t \asymp Q_2\\n_1\mid \breve{t}}}(t, \lambda_2)\Big)^{1/2} \Big(\sum_{n_1\asymp H}\sum_{\substack{   t \asymp Q_2\\n_1 \mid \breve{t}}}(t, \lambda_1)^2\\
	&\quad\quad \times \sum_{n_2\asymp N}v\Big(\frac{n_2}{N}\Big)\Big|\sum_{m}\sum_{\substack{p\asymp Q_1\\ p \text{ prime}\\ p\nmid h\lambda_1\lambda_2}}\Sigma_{h, m\overline{\lambda_1'}(t, \lambda_1), n_1(t, \lambda_2), \pm n_2\overline{\lambda_2'}}(pt)\Omega_{z, n_1, n_2}(m, pt) \Big|^2\Big)^{1/2} dz
%	\\&\ll \frac{x^{3/2+\varepsilon}}{|\lambda_1\lambda_2|^{1/2}\delta Q_1^{7/2}Q_2^{3}} \int_{z\ll \delta} \Big(\sum_{n_1\asymp H}\sum_{\substack{n_1\mid \frac{t}{(t, \lambda_2)}\\  t \asymp Q_2}} (t, \lambda_1\lambda_2) \sum_{n_2}v\Big(\frac{n_2}{N}\Big)\\
%	& \quad\quad\quad\quad \times \Big|\sum_{m \in \mathbb Z} \sum_{\substack{p\asymp Q_1\\ p \text{ prime}\\ p\nmid h\lambda_1\lambda_2}}\Sigma_{h, m\overline{\lambda_1'}(t, \lambda_1), n_1(t, \lambda_2), \pm n_2\overline{\lambda_2'}}(pt)\Omega_{z, n_1, n_2}(m, pt)\Big|^2\Big)^{1/2} dz
	\end{split}
\end{equation}
where $v$ is some non-negative smooth function with support on $[1/3, 3]$ and $v(x)=1$ on $[1/2,2]$.  %Note that for fixed $n_1, t$, the contribution from $n_2\asymp N$ is negligible unless $N\ll \frac{Q_1^3t^3|\lambda_2|}{(t, \lambda_2)^3n_1^2x}$. 
%Recall that $(h, \lambda_1, \lambda_2) = 1$, hence $\lambda_1, \lambda_2, h$ are pairwise coprime (otherwise the equation
%So far it was not necessary, but we observe that we may assume up front that $\lambda_1, \lambda_2, h$ are pairwise coprime, otherwise we can divide out a common gcd from the equation 
%$\lambda_1m - \lambda_2n = h$ cannot be satisfied. This allows us to write %$(t, \lambda_1)(t, \lambda_2)  \leq (t, \lambda_1\lambda_2)$. %We also recall the size condition on $N$ implicit in \eqref{Omega}. 
Recalling that $\lambda_1, \lambda_2$ are coprime, we can recast the above sum
\begin{equation}\label{n2sum}
	\begin{split}
	&\frac{x^{3/2+\varepsilon}}{|\lambda_1||\lambda_2|^{1/2}\delta LQ^{3/2}} \int_{z \ll \delta} \Big(\sum_{\substack{n_1\asymp H}}\sum_{n_2\asymp N}\frac{|A(n_1, n_2)|^2}{n_2}\sum_{g_2\mid \lambda_2}g_2\sum_{\substack{ \breve{t} \asymp Q_2/g_2\\n_1 \mid \breve{t}}} 1\Big)^{1/2} \Big(\sum_{n_1\asymp H}\sum_{\substack{   t \asymp Q_2\\n_1 \mid \breve{t}}}(t, \lambda_1)^2\\
	&\quad\quad  \times \sum_{n_2\asymp N}v\Big(\frac{n_2}{N}\Big)\Big|\sum_{m}\sum_{\substack{p\asymp Q_1\\ p \text{ prime}\\ p\nmid h\lambda_1\lambda_2}}\Sigma_{h, m\overline{\lambda_1'}(t, \lambda_1), n_1(t, \lambda_2), \pm n_2\overline{\lambda_2'}}(pt)\Omega_{z, n_1, n_2}(m, pt) \Big|^2\Big)^{1/2} dz\\
	&\ll \frac{x^{3/2+\varepsilon}}{|\lambda_1||\lambda_2|^{1/2}\delta Q_1^{7/2}Q_2^3} \int_{z \ll \delta} \Big( \sum_{g_2\mid \lambda_2}\sum_{\breve{t} \asymp Q_2/g_2}\sum_{g_1\mid (\breve{t}, \lambda_1)}g_1^2\sum_{n_1\mid \breve{t}}\sum_{n_2\asymp N}v\Big(\frac{n_2}{N}\Big)\\
	&\quad \quad  \times \Big|\sum_{m}\sum_{\substack{p\asymp Q_1\\ p \text{ prime}\\ p\nmid h\lambda_1\lambda_2}}\Sigma_{h, m\overline{\lambda_1'}g_1, n_1g_2, \pm n_2\overline{\lambda_2'}}(pg_2\breve{t})\Omega_{z, n_1, n_2}(m, pg_2\breve{t}) \Big|^2\Big)^{1/2} dz.
	\end{split}
\end{equation}

Expanding the square and changing the order of summation,  the $n_2$-sum becomes
\begin{equation}\label{n20}
\begin{split}
	&\sum_{m_1, m_2} \sum_{\substack{p_1, p_2 \asymp Q_1\\ p_1, p_2 \text{ prime}\\(p_1p_2, h\lambda_1\lambda_2)=1}}  \sum_{n_2}v\Big(\frac{n_2}{N}\Big)\Omega_{z, n_1, n_2}(m_1, p_1g_2\breve{t})\overline{\Omega_{z, n_1, n_2}(m_2, p_2g_2\breve{t})}\\&\times\Sigma_{h, m_1\overline{\lambda_1'}g_1, n_1g_2, \pm n_2\overline{\lambda_2'}}(p_1g_2\breve{t})
\overline{\Sigma_{h, m_2\overline{\lambda_1'}g_1, n_1g_2, \pm n_2\overline{\lambda_2'}}(p_2g_2\breve{t})}. 
\end{split}
\end{equation}
From \eqref{Omega}, we see that the above is negligible unless 
\begin{equation}\label{m1m2}
m_1, m_2\ll x^\varepsilon \frac{|\lambda_1|Q^2}{g_1^2x} \text{ and }N\ll x^\varepsilon\frac{|\lambda_2|Q^3}{g_2^3Hx}. 
\end{equation}
Applying Poisson summation  modulo $[p_1, p_2]g_2\breve{t}$, the inner sum becomes
 \begin{equation}\label{afterPoisson}
 	\begin{split}
&\frac{N}{[p_1, p_2]g_2\breve{t}}\sum_{n_2}\mathcal T\big(h,   m_1\overline{\lambda_1'}g_1, m_2\overline{\lambda_1'}g_1, n_1g_2, \pm n_2\overline{\lambda_2'},p_1, p_2,g_2\breve{t}\big) \\
&\times \int_{\mathbb R}v(x)\Omega_{z, n_1, Nx}(m_1, p_1g_2\breve{t})\overline{\Omega_{z, n_1, Nx}(m_2, p_2g_2\breve{t})} e\Big(- \frac{xn_2}{[p_1, p_2]g_2\breve{t}}\Big) dx
\end{split}
\end{equation}
 using the notation \eqref{defT}. We write $$t= g_2\breve{t}= t_1 t_2$$ with $t_1$ squarefree, $t_2$ squarefull and $(t_1, t_2) = 1$.  Note that we can choose representatives of $\overline{\lambda_1'}\, \ppmod{\qm}, \overline{\lambda_2'}\, \ppmod{\qn}$ such that $(\overline{\lambda_1'} \overline{\lambda_2'}, t)=1$. Then Lemma \ref{charT} tells us that $\mathcal{T}$ is bounded by 
 %\begin{align}
% 	\mathcal T(h, 1, -m_1, -m_2, n_1, n_2, p_1, p_2, t)=\sum_{x\ppmod {[p_1, p_2]t}}e\Big(\frac{\pm xn_2}{[p_1, p_2]t}\Big)\Sigma_{h, 1, -m_1, n_1,  x}(p_1t)\overline{\Sigma_{h, 1, -m_2, n_1, x}(p_2t)}.
% \end{align}
%Using Lemma \ref{charT}, this is bounded by 
\begin{align}
	\begin{cases}
		p_1^{2}(m_1-m_2, p_1)t^{3+\varepsilon}, & p_1=p_2, n_2=0,\\
		p_1^{3}t^{5/2+\varepsilon}(hn_2p_1, t_1)^{1/2}t_2^{1/2}, & p_1=p_2, n_2\not=0,\\
		p_1^{3/2}p_2^{3/2}t^{5/2+\varepsilon}(hn_2, t_1)^{1/2}t_2^{1/2}, & p_1\not=p_2, n_2\not=0,\\
		0, & p_1 \not = p_2, n_2 = 0.
	\end{cases}
\end{align}
 Integration by parts shows that the contribution from $|n_2|\geq N_2 := x^\varepsilon\frac{[p_1, p_2]g_2\breve{t}}{N}$ is negligible, 
so that \eqref{afterPoisson} is bounded by
\begin{equation}\label{insert}
\begin{split}
&\Big(\frac{N}{p_1t} p_1^{2}(m_1-m_2,p_1)t^{3+\varepsilon}+\frac{N}{p_1t}\sum_{1\leq |n_2|\leq N_2}	p_1^{3}t^{5/2+\varepsilon} (hn_2p_1, t_1)^{1/2} t_2^{1/2}\Big)\mathbf{1}_{p_1=p_2}\\
& \quad\quad\quad\quad\quad +\frac{N}{p_1p_2t}\sum_{1\leq |n_2|\leq N_2}p_1^{3/2}p_2^{3/2}t^{5/2+\varepsilon} (hn_2, t_1)^{1/2}t_2^{1/2}\mathbf{1}_{p_1\not=p_2}\\
& \ll x^\varepsilon\Big( Np_1(m_1-m_2, p_1)t^{2}\mathbf{1}_{p_1=p_2}+p_1^{3}t^{5/2} (hp_1, t_1 )^{1/2}  t_2^{1/2}\mathbf{1}_{p_1=p_2}%\\
%& \quad \quad\quad \quad 
+ p_1^{3/2}p_2^{3/2}t^{5/2}(h, t_1)^{1/2}t_2^{1/2}%\mathbf{1}_{p_1\not=p_2}
\Big).
\end{split}
\end{equation}
We assume
\begin{equation}\label{Qlower}
	Q \geq (x|\lambda_1|)^{1/2}
\end{equation}
 and sum this first over $m_1, m_2$, keeping in mind \eqref{m1m2} and the fact that $|\lambda_1|Q^2/g_1^2x \geq 1$, and then over $p_1, p_2$.  In this way we bound \eqref{n20} by
\begin{align}
\sum_{p_1 \asymp Q_1} \frac{x^\varepsilon |\lambda_1|Q^2}{g_1^2x}Np_1^2t^2 &+\frac{x^\varepsilon |\lambda_1|^2Q^4}{g_1^4x^2}\Big(\sum_{p_1 \asymp Q_1} \Big(Np_1t^2 +p_1^3t^{5/2}(hp_1, t_1)^{1/2}t_{2}^{1/2} \Big)\\
&\quad \quad\quad \quad + \sum_{p_1, p_2 \asymp Q}p_1^{3/2}p_2^{3/2}t^{5/2}(h, t_1)^{1/2}t_2^{1/2} \Big),
\end{align}
where the first term is the contribution of $m_1 = m_2$ in the first term of \eqref{insert}.

Substituting this back into \eqref{n2sum}, we see that $S_1(N, H)$ is bounded by 
\begin{equation*}%}\label{s1final}
\begin{split}
	&\ll    \frac{x^{3/2+\varepsilon}}{|\lambda_1||\lambda_2|^{1/2}Q_1^{7/2}Q_2^3} \Big( \sum_{g_2\mid \lambda_2}\sum_{  \breve{t} \asymp Q_2/g_2} \sum_{g_1\mid (\breve{t}, \lambda_1)} g_1^2\Big(1 + \frac{H^2g_2^3Nx}{|\lambda_2|Q^3}\Big)^{-A} \frac{|\lambda_1|^2 Q^4}{g_1^4 x^2} \Big[\sum_{p_1\asymp Q_1}\frac{g_1^2x}{|\lambda_1|Q^2 }(Np_1^2 t^2 )\\ &\quad \quad\quad\quad +
	\sum_{p_1\asymp Q_1}\big(Np_1t^2+p_1^3t^{5/2}(hp_1, t_1)^{1/2}t_2^{1/2}\big)+ \sum_{p_1, p_2\asymp Q_1}p_1^{3/2}p_2^{3/2}t^{5/2}(h, t_1)^{1/2}t_2^{1/2}\Big]\Big)^{1/2}\\
%	&\frac{x^{1/2+\varepsilon}}{Q_1^{3/2}Q_2}\Big(\sum_{n_1\sim H}\sum_{n_1\mid t\sim Q_2} \Big(\sum_{p_1\sim Q_1}\Big( Np_1^{2}t^2\sqrt{t_2}+p_1^{5/2}t^{5/2}(h, \frac{p_1t}{t_2})^{1/2}t_2^{1/2}\Big)+\sum_{p_1, p_2\sim Q_1}p_1^{3/2}p_2^{3/2}t^{5/2}(h, \frac{t}{t_2})^{1/2}t_2^{1/2}\Big)\Big)^{1/2}\\
&\ll    \frac{x^{1/2+\varepsilon}}{|\lambda_2|^{1/2}Q_1^{3/2}Q_2} \Big(  \sum_{g_2\mid \lambda_2}\sum_{ \breve{t} \asymp Q_2/g_2}\sum_{g_1\mid (\breve{t}, \lambda_1)}   \Big[\Big(\frac{xQ_1}{|\lambda_1|Q^2}+1\Big)\frac{|\lambda_2|Q^3}{g_2^3H^2x}Q_1^2t^2 %Q_1^{4}t^{5/2}(h , t_1)^{1/2}t_2^{1/2} 
+ Q_1^5t^{5/2}(t_1, h)^{1/2}t_2^{1/2} \Big] \Big)^{1/2}\\
&\ll    \frac{x^{1/2+\varepsilon}}{|\lambda_2|^{1/2}Q_1^{3/2}Q_2} \Big( \sum_{t \asymp Q_2 }   \Big[\Big(\frac{xQ_1}{|\lambda_1|Q^2}+1\Big)\frac{|\lambda_2|Q^3}{x}Q_1^2t^2 %Q_1^{4}t^{5/2}(h , t_1)^{1/2}t_2^{1/2} 
+ Q_1^5t^{5/2}(t_1, h)^{1/2}t_2^{1/2} \Big] \Big)^{1/2}\\
%& \ll \frac{x^{1/2+\varepsilon}}{Q_1^{3/2}Q_2}\Big(\sum_{n_1\asymp H}\Big(\frac{Q_1x}{Q^2}+1\Big)NQ_1^{2}\frac{Q_2^3}{n_1}+(h, n_1)^{1/2}Q_1^{4}\frac{Q_2^{7/2}}{\sqrt{n_1}}+(h, n_1)^{1/2}Q_1^5\frac{Q_2^{7/2}}{\sqrt{n_1}}\Big)^{1/2}\\
& \ll \frac{x^{1/2+\varepsilon}}{|\lambda_2|^{1/2}Q_1^{3/2}Q_2}\Big(\frac{|\lambda_2|}{|\lambda_1|}Q^2+\frac{|\lambda_2|Q^3}{x}Q_1^{2}Q_2^3+Q_1^5Q_2^{7/2}\Big)^{1/2} \\
& \ll  x^{\varepsilon}\Big(\frac{x^{1/2 }Q}{|\lambda_1|^{1/2}Q_1^{1/2}}+\frac{Q^2}{Q_1}+\frac{x^{1/2 }Q }{|\lambda_2|^{1/2}Q_2^{1/4}}\Big)
\end{split}
\end{equation*}
 %(recall \eqref{squarefull} in the third step) and so
%\begin{equation*}%\label{s1final}
%S_1(N, H) \ll  x^{\varepsilon}\Big(\frac{x^{1/2 }Q}{|\lambda_1|^{1/2}Q_1^{1/2}}+\frac{Q^2}{Q_1}+\frac{x^{1/2 }Q }{|\lambda_2|^{1/2}Q_2^{1/4}}\Big)
%%\ll  x^{\varepsilon}  \Big(\frac{x^{1/2 }Q}{|\lambda_1\lambda_2|^{1/2}Q_1^{1/2}}+\frac{Q^2}{|\lambda_1\lambda_2|^{1/2}Q_1 }+\frac{x^{1/2 }Q}{|\lambda_1\lambda_2|^{1/4}Q_2^{1/4}}\Big)
%\end{equation*}
under the assumptions \eqref{q1h} and \eqref{Qlower}.

On the other hand, if  $H \gg Q_1$, we apply \eqref{trivialbound}  to see that
$$S_1(N, H) %\ll x^{\varepsilon} \Big(\frac{x}{H^{1/2}} + \frac{Q^2}{H}\Big) 
\ll x^{\varepsilon} \Big(\frac{x}{ Q_1} + \frac{Q^2}{Q_1}\Big).$$
This is dominated by the previous bound under \eqref{Qlower}, and so 
we obtain 
the final bound 
\begin{equation}\label{s1final}
S_1 \ll x^{\varepsilon}  \Big( \frac{x^{1/2 }Q}{|\lambda_1|^{1/2}Q_1^{1/2}}+\frac{Q^2}{Q_1 }+\frac{x^{1/2 }Q}{|\lambda_2|^{1/2}Q_2^{1/4}}\Big),
\end{equation}
provided that  \eqref{Qlower}   holds.

%provided that \eqref{q1h} holds, and we recall the bound \eqref{trivialbound} in any case. Finally we recall the size restriction \eqref{sizes}. 

\subsection{The endgame} 
%We recall that we may assume that $(\lambda_1, \lambda_2) = 1$.  
%If $ |\lambda_2| \geq x^{1/42}$, then the result is trivial. If $|\lambda_1| \geq x^{1/21}$, then 
By the Cauchy--Schwarz inequality 
we have (recall that $(\lambda_1, \lambda_2) = 1$)
\begin{displaymath}
\begin{split}
& \sum_{\lambda_1m - \lambda_2n = h} A(n, 1) \tau(m) W_0\Big(\frac{|\lambda_1| m}{x}\Big)W\Big(\frac{|\lambda_2| n}{x}\Big) \ll x^{\varepsilon} \sum_{\substack{\lambda_2 n \equiv h\, (\text{mod } |\lambda_1|)\\ n \ll x/|\lambda_2|}} |A(n, 1)| \\
&\ll   x^{\varepsilon} \Big( \frac{x}{|\lambda_2|}\Big)^{1/2} \Big(1 + \frac{x}{|\lambda_2| |\lambda_1|}\Big)^{1/2} \ll x^{\varepsilon} \Big(\frac{x}{ |\lambda_2| |\lambda_1|^{1/2}} + \frac{x^{1/2}}{|\lambda_2|^{1/2}}\Big).
\end{split}
\end{displaymath}
 
Hence we can assume $|\lambda_1| \ll x^{1/21}$ %, $|\lambda_2|\ll x^{1/42}$
 (otherwise we use the previous trivial bound), in which case we choose $C, C_0, Q_1, Q_2$ as
\begin{equation}\label{choice}
C_0 = x^{19/42-\eta}, \quad C = x^{23/42+\eta}, \quad Q_1= x^{4/21}, \quad Q_2 = q^{8/21}
\end{equation}
for some fixed, but arbitrarily small $\eta > 0$, 
so $Q = x^{12/21}$. For the present situation, we could  choose $\eta = 0$, but 
%We have some flexibility in the choice of $C$, 
these values are designed to work also for the proof of Theorem \ref{thm1a}. 

With this choice, we see that \eqref{C0Ccondition} and \eqref{Qlower} hold and from \eqref{s2final} and \eqref{s1final} we find that
$$S_1, S_2  \ll x^{41/42 +\eta+ \varepsilon}.$$
%If $H \gg Q_1$, we apply \eqref{trivialbound}  to see that
%$$S_1(N, H) %\ll x^{\varepsilon} \Big(\frac{x}{H^{1/2}} + \frac{Q^2}{H}\Big) 
%\ll x^{\varepsilon} \Big(\frac{x}{Q_1^{1/2}} + \frac{Q^2}{Q_1}\Big) \ll x^{20/21 + \varepsilon}.$$
%On the other hand, if \eqref{q1h} holds then by combining \eqref{s1final} shows
%\begin{displaymath}
%\begin{split}
%S_1(N, H) & \ll %x^{\varepsilon} \Big(\frac{x^{1/2}Q}{Q_1^{1/2}}+\frac{Q^2}{Q_1H} + \min\Big( \frac{x^{1/2} Q H^{1/4}}{Q_2^{1/4}}, \frac{x}{H^{1/2}} + \frac{Q^2}{H}\Big)\Big)\\
%& \ll x^{\varepsilon} \Big(\frac{x^{1/2}Q}{Q_1^{1/2}}+\frac{Q^2}{Q_1H}  + x^{2/3} Q_1^{2/3} Q_2^{1/2}, x^{2/5} Q_1^{6/5} Q_2\Big) \ll 
%x^{41/42+\varepsilon}.
%\end{split}
%\end{displaymath}
Since $\eta$ can be arbitrarily small, this completes the proof of Theorem \ref{thm1}.

\section{Proof of Theorem \ref{thm1a}}

We indicate the modifications of the previous proof necessary for the proof of Theorem \ref{thm1a}. The only difference is that the classical Voronoi summation formula (Lemma \ref{vortau}) is replaced with Lemma \ref{vorconv}. This has the same structure except that $rq = 0$ can come from the three sources $r=q=0$, $r\not=q = 0$ and $q\not=r = 0$. In our application the ``half-diagonal'' terms $rq = 0$ but $(r, q) \not=(0, 0)$ will not play a major role, since we may assume that $A, B$ in $\tau_{A, B}(m)$   are roughly of equal size, otherwise we  apply Voronoi summation as follows.  We have
\begin{displaymath}
\begin{split}
\mathcal{S}_{A, B}(x) := &\sum_{\lambda_1 m- \lambda_2n = h} A(n, 1) \tau_{A, B}(m)   W\Big(\frac{|\lambda_2| n}{x}\Big)\\
& = \sum_{a} v_1\Big(\frac{a}{A}\Big)  \sum_{ \lambda_2n \equiv  -h\, (\text{mod } \lambda_1a)} A(n, 1)W\Big(\frac{|\lambda_2| n}{x}\Big) v_2\Big(\frac{\lambda_2 n + h}{\lambda_1 a B}\Big).
\end{split}
\end{displaymath}
Since $\lambda_1, \lambda_2, h$ are pairwise coprime, the congruence is void unless $(\lambda_2, a) = 1$. Hence we obtain
 \begin{displaymath}
\begin{split}
&\sum_{(a, \lambda_2) = 1} v_1\Big(\frac{a}{A}\Big)  \sum_{  n \equiv  -\overline{\lambda_2}h\, (\text{mod } \lambda_1a)} A(n, 1)W\Big(\frac{|\lambda_2| n}{x}\Big) v_2\Big(\frac{\lambda_2 n + h}{\lambda_1 a B}\Big)\\
& = \sum_{\substack{\lambda_1\mid a\\(a, \lambda_2) = 1}} v_1\Big(\frac{a}{|\lambda_1|A}\Big) \frac{1}{a}\sum_{\alpha \mid a} \sum_{\substack{b\, (\text{mod } \alpha)\\ (b, \alpha) = 1}}  \sum_{ n} A(n, 1)e\Big( \frac{(n + \overline{\lambda_2}h)b}{\alpha}\Big)W\Big(\frac{|\lambda_2| n}{x}\Big) v_2\Big(\frac{\lambda_2 n + h}{a B}\Big)
\end{split}
\end{displaymath}
 A standard application of the Voronoi summation formula (Lemma \ref{vorGL3}, cf.\ also \eqref{intbyparts}) as before bounds this by
 $$\sum_{\substack{(a, \lambda_2) = 1\\ \lambda_1\mid a \asymp |\lambda_1|A}} \frac{1}{a} \sum_{\alpha \mid a}  \frac{1}{\alpha^2} \sum_{\pm} \sum_{n_2} \sum_{n_1 \mid \alpha} |\Sigma_{\overline{\lambda_2}h, 0, n_1, \pm n_2}(\alpha)| n_1 |A(n_1, n_2) | \frac{x}{|\lambda_2|}\Big(\frac{n_1^2n_2x}{|\lambda_2|\alpha^3} \Big)^{-1/2} \Big(1 +\frac{n_1^2n_2x}{|\lambda_2|\alpha^3} \Big)^{-K} $$ 
using the notation  \eqref{Sigma}. By Lemma \ref{char1} (with the usual notation $\alpha = \alpha_1\alpha_2$ where $\alpha_2$ is the squarefull part of $\alpha$) we obtain for any $K > 1$ the bound
 \begin{displaymath}
\begin{split}
& x^{1/2+\varepsilon} \sum_{ \lambda_1\mid a \asymp |\lambda_1|A} \frac{1}{a} \sum_{\alpha \mid a}  \Big(\frac{ \alpha  \alpha_2}{|\lambda_2|}\Big)^{1/2}\sum_{n_2} \sum_{n_1 \mid \alpha}  \frac{|A(n_1, n_2) | }{n_2^{1/2}}  \Big(1 +\frac{n_1^2n_2x}{|\lambda_2|\alpha^3} \Big)^{-K} \\& \ll   x^\varepsilon\sum_{\lambda_1\mid a \asymp |\lambda_1|A} \frac{1}{a} \sum_{\alpha \mid a}    \alpha^2  \alpha^{1/2}_2  \ll x^\varepsilon |\lambda_1|^{3/2}A^2.
\end{split}
\end{displaymath}
Exchanging the roles of $A$ and $B$, we obtain the bound 
\begin{equation}\label{prelim}
  \mathcal{S}_{A, B}(x) \ll  x^\varepsilon |\lambda_1|^{3/2}\min(A^2, B^2). 
\end{equation}   
    
\medskip

After this preliminary bound,  we follow the proof of Theorem \ref{thm1} to write $\mathcal{S}_{A, B}(x) = S_1 + S_2$ 
%\begin{align}
%	\sum_{\lambda_1m-\lambda_2n=h}A(n, 1)W\Big(\frac{|\lambda_2|n}{x}\Big)\tau_{A,B}(m)=S_1+S_2
%\end{align}
where 
\begin{displaymath}
	\begin{split}
		S_1 &= \sum_{n, m} A(n, 1) W\Big( \frac{|\lambda_2|n}{x}\Big) \tau_{A,B}(m)  \int_0^1   \chi(\alpha) e((\lambda_1m-\lambda_2n-h)\alpha) d\alpha,\\
		S_2 &= \sum_{n, m} A(n, 1) W\Big( \frac{|\lambda_2|n}{x}\Big) \tau_{A,B}(m)  \int_0^1 \big(   1 - \chi(\alpha)\big)e((\lambda_1m-\lambda_2n-h)\alpha) d\alpha.
	\end{split}
\end{displaymath}
Note that Lemma \ref{vorconv} together with $AB\asymp x/|\lambda_1|$ gives
\begin{equation}
	\begin{split}
	&\sum_{m}\tau_{A,B}(m)e\Big(\lambda_1m\Big(\frac{b}{c}+z\Big)\Big)\\&=\frac{1}{\cm}\sum_{a,b}e\Big(\frac{-ab\overline{b\lambda_1'}}{\cm}\Big)\int_{\RR^2}v_1\Big(\frac{u}{A}\Big)v_2\Big(\frac{w}{B}\Big)e(uw\lambda_1z)e\Big(\frac{ua+wb}{\cm}\Big)dudw\\
	& \ll \frac{AB}{\cm}\sum_{a,b}\Big(1+\frac{Aa}{\cm}+\frac{Bb}{\cm}\Big)^{-K}+\frac{1}{1+x|z|}\Big(1+\frac{Aa}{\cm(1+x|z|)} +\frac{Bb}{\cm(1+x|z|)}\Big)^{-K}\\
	& \ll %\frac{AB}{\cm}\Big(1+ \frac{\cm}{A}+ \frac{\cm}{B}+\frac{\cm^2(1+x|z|)}{AB}\Big)\ 
	 \frac{AB}{\cm}+A+B+\cm(1+x|z|)
	\end{split}
\end{equation}
for any $K>0$ as an analogue of \eqref{msum}. Thus we have 
\begin{equation}
	\begin{split}
	\max_{C_0\leq c\leq C}\max_{\substack{ b\ppmod c\\ (b,c)=1}}\max_{|z|\ll (cC)^{-1}}\Big|\sum_{m}\tau_{A,B}(m)e\Big(\lambda_1m \Big(\frac{b}{c}+z\Big)\Big)\Big|\ll  x^\varepsilon \Big(\frac{x}{C_0}+A+B+C\Big),
	\end{split}
\end{equation}
so that the estimate corresponding to \eqref{s21} becomes
\begin{align}
	S_{2,1}\ll x^\varepsilon \frac{1}{ Q\delta^{1/2}}\Big(\frac{x}{|\lambda_2|}\Big)^{1/2}\Big(\frac{x}{C_0}\Big(1+\frac{(A+B)C_0}{x}\Big)+C\Big).
\end{align}
For $c\leq C_0$ and $z\leq (cC)^{-1}$
 we see that the contirbution from $ab\not=0$ is negligible unless 
\begin{align}
	a\ll x^\varepsilon \frac{\cm(1+x|z|)}{A}\ll x^\varepsilon \Big(\frac{C_0}{A}+\frac{x}{AC}\Big), \quad 	b\ll x^\varepsilon \frac{\cm(1+x|z|)}{B}\ll x^\varepsilon\Big(\frac{C_0}{B}+\frac{x}{BC}\Big).
\end{align}
Thus, if we choose 
\begin{align}\label{ABC}
	C_0\leq x^{-\eta} \min(A,B), \quad C\geq\max (x/A, x/B)x^{\eta}=|\lambda_1|\max(A,B)x^{\eta}
\end{align}
(since $AB = x/|\lambda_1|$) for some very small $\eta > 0$, then the contribution from $ab\not=0$ is negligible. Moreover, the contribution from $a=0$ or $b=0$ is $\ll A+B\ll  AB/C_0$, thus dominated by the contribution from $a=b=0$. As in \eqref{z} we choose $z\ll x^{-1+\varepsilon}$. Therefore, under the assumptions in \eqref{ABC}, we see the same bounds in \eqref{s200} and \eqref{s201bound} hold, and we conclude 
\begin{equation}\label{s2final1}
	S_{2}\ll x^{1+\varepsilon }\Big(\frac{x}{|\lambda_2|^{1/2}QC_0}  + \frac{\max(A, B)}{|\lambda_2|^{1/2} Q}%\Big(1+\frac{(A+B)C_0}{x}\Big)
	+\frac{C}{|\lambda_2|^{1/2}Q}+\frac{C_0^2}{|\lambda_1|^{1/2}x}+\frac{C_0^3}{|\lambda_1|^{1/2}xQ_2}+\frac{C_0}{|\lambda_1|^{1/2}Q}\Big)
\end{equation}
as an analogue of \eqref{s2final}. 

\medskip

For $S_1$, we see that for $q\leq Q, z\leq x^{-1+\varepsilon}$, the main terms in the dual summation when $ab=m=0$ are of size
\begin{align}
	\frac{AB}{\qm}\Big(1+\frac{\qm}{A}+\frac{\qm}{B}\Big)\ll \frac{AB}{\qm}\Big(1+\frac{Q}{A}+\frac{Q}{B}\Big).
\end{align}
so that the trivial bound becomes (cf. \eqref{trivialbound} where the first term comes from $m=0$)
\begin{align}
	S_1(N, H)\ll x^\varepsilon \Big(\frac{x}{H}\Big(1+\frac{Q}{A}+\frac{Q}{B}\Big)+\frac{Q^2}{H}\Big).
\end{align}
If we choose 
\begin{align}\label{QAB}
	Q\gg |\lambda_1| \max(A, B),
\end{align}
then the arguments after the Cauchy steps remain the same since we have $\qm^2/(x/|\lambda_1|)=\qm^2/AB\geq 1$ terms for the dual variables $a,b$. 

Thus the analogue of \eqref{s1final} becomes  
\begin{equation}\label{s1final1}
	S_1\ll x^\varepsilon \Big( \frac{x^{1/2 }Q}{|\lambda_1|^{1/2}Q_1^{1/2}}+\frac{Q^2}{Q_1 }+\frac{x^{1/2 }Q}{|\lambda_2|^{1/2}Q_2^{1/4}}\Big)+x^\varepsilon \Big(\frac{xQ}{Q_1 \min(A, B)}%\Big(1+\frac{Q}{A}+\frac{Q}{B}\Big)
	+ \frac{Q^2}{Q_1}\Big).
\end{equation}

\medskip

With \eqref{prelim}, \eqref{s2final1} and \eqref{s1final1} we can conclude the proof in a similar way. Again we can assume without loss of generality $|\lambda_1| \leq x^{1/21}$. % , $|\lambda_2| \leq x^{1/42}$. 
Moreover, by \eqref{prelim} we can assume $\min(A, B) \geq x^{19/42}$ and hence $|\lambda_1| \max(A, B) \leq x^{23/42}$. Now we make the same choice as in \eqref{choice}, which satisfies \eqref{ABC} and \eqref{QAB}, getting again
%$Now we choose
%$C_0=x^{19/42}, Q_1=x^{4/21}, Q=x^{12/21}, C=x^{1/2+1/41}
%|\lambda_1|<Q$, we see that 
\begin{align}
	S_1, S_2\ll x^{41/42+\eta+\varepsilon} .%+x^{40/42+\varepsilon}(\frac{Q}{A}+\frac{Q}{B}).
\end{align}
%If we choose $x^{1/2-1/82}<A,B< x^{1/2+1/82}$, then we see that \eqref{ABC} and \eqref{QAB} hold and 
%\begin{align}
%	S_1, S_2\ll x^{41/42+\varepsilon}.
%\end{align}
%If $\min (A,B)\ll x^{1/2-1/82}$, we see that we also have a bound $\min (A^2, B^2)\ll x^{41/42+\varepsilon}$. 
 This completes the proof. 
 
 \begin{remark}\label{rem2} We have assumed that the weight functions $W, W_0, v_1, v_2$ in Theorems \ref{thm1} and \ref{thm1a} are fixed. During the proofs they are subject to finitely many integrations by parts, and hence it is clear that the implicit constants depend on some suitable Sobolev norms of these weight functions. 
 \end{remark}

 \section{Proof of Theorem \ref{thm2}}

 Let $F$ be a cusp form for  the group ${\rm SL}(3, \Bbb{Z})$ and $\chi$ an even  primitive Dirichlet character modulo $q$. We start with a standard approximate functional equation \cite[Theorem 5.3, Proposition 5.4]{IK}
 $$L(1/2, F \times \chi) = \sum_{n} \frac{A(n, 1) \chi(n)}{\sqrt{n}} V\Big( \frac{n}{q^{3/2} X}\Big) + \frac{\tau(\chi)^3}{q^{3/2}}\sum_{n} \overline{\frac{A(n, 1) \chi(n)}{\sqrt{n}} }V\Big( \frac{nX}{q^{3/2} }\Big)$$
 where $$\tau(\chi) = \sum_{h\, (\text{mod } q)} \chi(h) e(h/q)$$ is the standard Gau{\ss} sum, $X > 0$ is a parameter at our disposal and $V$ is a smooth function depending on $F$ satisfying
 $$y^j V^{(j)}(y) \ll_{j, A}  (1 + y)^{-A}$$
 for any $j, A \geq 0$  and
\begin{equation}\label{v0}
V(y) = 1 + O( y^{1/21}), \quad y \rightarrow 0
\end{equation}
 using the constant $5/14$ \cite[Proposition 1]{KS} towards the Ramanujan conjecture for the archimedean Langlands parameters of $F$. (The exponent $\frac{1}{21} = \frac{1}{3}(\frac{1}{2} - \frac{5}{14})$ suffices for our purpose, but could be improved by modifying \cite[Proposition 5.4]{IK}.) We will optimize $X$ later, for now we assume 
 $$Q^{1/2} \leq X \leq Q^{2/3}.$$
   
 If $f$ is any function on characters, then by M\"obius inversion we have 
  $$\sum_{\substack{\chi \, (\text{mod } q)\\ \chi \text{ primitive, even}}} f(\chi) =  \frac{1}{2} \sum_{\substack{\chi \, (\text{mod } q)\\ \chi \text{ primitive}}} (1 + \chi(-1)) f(\chi)  = \frac{1}{2}\sum_{d \mid q} \mu\Big(\frac{q}{d}\Big) \sum_{\chi\, (\text{mod } d)} (1 + \chi(-1)) f(\chi^{\ast}) $$
 where $\chi^{\ast}$ denotes the character modulo $q$ induced from $\chi$. Hence
 $$\mathcal{L}(Q) := \sum_q W\Big(\frac{q}{Q}\Big) \sum_{\substack{\chi \, (\text{mod } q)\\ \chi \text{ primitive, even}}}L(1/2, F \times \chi) = T_1 + T_2$$
 where
 $$T_1 = \frac{1}{2}\sum_{\pm}  \sum_q W\Big(\frac{q}{Q}\Big)\sum_{d \mid q} \mu\Big(\frac{q}{d}\Big) \phi(d)  \sum_{\substack{n \equiv \pm 1\, (\text{mod } d)\\ (n, q) = 1}} \frac{A(n, 1) }{\sqrt{n}} V\Big( \frac{n}{q^{3/2} X}\Big)   $$
 and
 $$T_2 =   \sum_q W\Big(\frac{q}{Q}\Big) \sum_{(n, q) = 1}\overline{\frac{ A(n, 1)}{\sqrt{n}} } V\Big( \frac{nX}{q^{3/2} }\Big) \mathcal{K}(n; d, q)$$
 with
 $$ \mathcal{K}(n; d, q) = \frac{1}{2} \sum_{d \mid q}\mu\Big(\frac{q}{d}\Big)  \sum_{\chi\, (\text{mod } d)} (1 + \chi(-1)) \frac{\tau(\chi^{\ast})^3}{q^{3/2}} \bar{\chi}(n).$$
 By \cite[Lemma 3.1]{IK} we have
\begin{displaymath}
\begin{split}
\mathcal{K}(n; d, q) & = \frac{1}{2q^{3/2}} \sum_{d \mid q}\mu\Big(\frac{q}{d}\Big)  \sum_{\chi\, (\text{mod } d)} (1 + \chi(-1))\Big(  \mu\Big( \frac{q}{d}\Big) \chi\Big(\frac{q}{d}\Big) \tau(\chi) \Big)^3\bar{\chi}(n)\\
& =  \frac{1}{2q^{3/2}} \sum_{\substack{d \mid q\\ (d, \frac{q}{d}) = 1}} \mu^2\Big(\frac{q}{d}\Big)\sum_{\pm}  \sum_{\chi\, (\text{mod } d)} \sum_{h_1, h_2, h_3\, (\text{mod } d)} \chi\Big(\pm  h_1h_2h_3\bar{n} \frac{q^3}{d^3}\Big) e\Big( \frac{h_1+h_2+h_3}{d}\Big) \\
& =  \frac{1}{2q^{3/2}} \sum_{\substack{d \mid q\\ (d, \frac{q}{d}) = 1}} \mu^2\Big(\frac{q}{d}\Big) \phi(d) \sum_{\pm}    \sum_{ \substack{h_1, h_2\, (\text{mod } d)\\ (h_1h_2, d) = 1}}  e\Big( \frac{h_1+h_2\pm \overline{h_1h_2} n \overline{q^3/d^3}}{d}\Big). 
\end{split}
\end{displaymath}
By Deligne's bound for hyper-Kloosterman sums, we obtain
$$\mathcal{K}(n; d, q)  \ll q^{1/2 + \varepsilon},$$
and so by trivial estimates
\begin{equation}\label{T2bound}
	T_2 \ll \frac{Q^{9/4+\varepsilon}}{X^{1/2}}.
\end{equation}

For $T_1$, we single out the contribution $\pm = +$, $n=1$ which equals
\begin{equation}\label{mainterm}
\begin{split}
& \frac{1}{2} \sum_q W\Big( \frac{q}{Q}\Big) \sum_{d \mid q} \mu\Big(\frac{q}{d}\Big)\phi(d)\Big(1 + O\Big(\frac{1}{(Q^{3/2}X)^{1/21}}\Big)\Big)\\
& = \frac{1}{2} \int_{(3)}  \sum_{d, q}  \frac{\mu(q)\phi(d)}{(dq)^s}   \widetilde{W}(s) Q^s \frac{ds}{2\pi i} + O\Big( \frac{Q^2}{(Q^{3/2}X)^{1/21}}\Big)\\
& =  \frac{1}{2} \int_{(3)} \frac{\zeta(s-1)}{\zeta(s)^2}   \widetilde{W}(s) Q^s \frac{ds}{2\pi i} + O\Big( \frac{Q^2}{(Q^{3/2}X)^{1/21}}\Big) \\
& = \frac{\widetilde{W}(2) }{2 \zeta(2)^2} Q^2 + O\Big( Q^{3/2} + \frac{Q^2}{(Q^{3/2}X)^{1/21}}\Big)
\end{split}
\end{equation}
by a standard contour shift argument. 

%In the remaining portion  %we note that we always have $n \gg d$. 
We now make a number of technical adjustments in preparation for applying Theorem \ref{thm1a}. 

We open the Euler $\phi$-function by writing $\phi(d) = \sum_{d_1d_2 = d} \mu(d_2) d_1$, and we write $q = d_1d_2 d'$. Next  we apply a smooth partitions of unity to the $d_1$-sum and the $n$-sum by attaching weight functions $$v\Big(\frac{n}{N}\Big) \Big(\frac{n}{N}\Big)^{1/2} v\Big(\frac{d_1}{D}\Big) \Big(\frac{d_1}{D}\Big)^{-1}$$ where
\begin{equation}\label{Nbound}
	N \ll Q^{3/2+\varepsilon}X
\end{equation}
 up to a negligible error. %We also localize $d_2 \asymp D_2$. 
 We open the existing weight function  $V(n/q^{3/2}X)$ by Mellin inversion as follows %e.g.\
%and localize in particular $n \asymp N$ with
%$$N \ll Q^{3/2 + \varepsilon}X,$$
%up to a negligible error, by attaching a factor $w(n/N)$ where $w$ is supported on $[1, 2]$.  We write
$$V\Big( \frac{n}{q^{3/2}X}\Big) = \int_{(0)}  \Big(\frac{N}{Q^{3/2}}\Big)^{-s} \widetilde{V}(s) \Big(\frac{nQ^{3/2}}{Nq^{3/2}X}\Big)^{-s}   \frac{ds}{2\pi i}.$$
%and similarly for $W(q/Q)$. 
Since $\widetilde{V}(s) \ll (1 + |s|)^{-A}$, we can truncate the integral at $|\Im s| \leq Q^{\varepsilon}$ at the cost of a negligible error, pull it outside and in this way separate the variables $n, q$ at essentially no cost. By slight abuse of notation, we replace $v(x)$ with $v(x) x^{- s}$  (without changing the notation). By Remark \ref{rem2} this only results in additional $Q^{\varepsilon}$-powers. 

Thus we are left with
 bounding 
$$T^{\pm}_1 :=  \frac{D}{N^{1/2}}  \sum_{d_2, d'} \Big|\sum_{d_1} W\Big( \frac{d_1d_2d'}{Q}\Big)v\Big(\frac{d_1}{D}\Big) \underset{\substack{n \equiv \pm 1\, (\text{mod } d_1d_2)\\ (n, d') = 1}}{\left.\sum\right.^{\ast}}  A(n, 1)  v\Big( \frac{n}{N}\Big)\Big|$$
(up to an outside integration of length $Q^{\varepsilon}$), where   $\sum^{\ast}$ indicates that in $+$-term the summand $n=1$ has been removed. At this point we record a trivial bound: we write $n = \pm 1 + d_1d_2 r$ and note the removal of the $n=1$ term implies $n\gg d_1d_2$, so $r \asymp N/d_1d_2$. Thus we obtain
\begin{equation*}%\label{T1trivial}
\begin{split}
T^{\pm}_1&\ll \frac{D}{N^{1/2}}\sum_{ d_2, d'} \sum_{d_1} W\Big( \frac{d_1d_2d'}{Q}\Big)v\Big(\frac{d_1}{D}\Big) \sum_{r \asymp N/d_1d_2} |A(1 + d_1d_2r, 1)| \\
& \ll \frac{D}{N^{1/2}} N^{1+\varepsilon} \frac{Q}{D} %\ll \frac{Q^{1+\varepsilon}N^{1/2}}{D_2}
 \ll Q^{1+\varepsilon}N^{1/2}.
\end{split}
\end{equation*}
This is more than sufficient for $N \leq 10$, say, and for larger $N$, which we assume from now on, the removal of the $n=1$ term is invisible; hence we remove the asterisk from the sum. 

Finally we remove the coprimality condition $(n, d') = 1$. Using the Hecke relations \cite[(2.2)]{BL}, we write the $n$-sum as
\begin{displaymath}
\begin{split}
&\sum_{ f\mid d' } \mu(f) \underset{ fn \equiv \pm 1\, (\text{mod } d_1d_2)}{\sum}  A(fn, 1) v\Big( \frac{fn}{N}\Big)\\
&= \sum_{ f_1\mid f_2 \mid f\mid d' }\mu(f) \mu(f_1)\mu(f_2) A\Big(\frac{f}{f_2}, \frac{f_2}{f_1}\Big) \underset{ f_1f_2fn \equiv \pm 1\, (\text{mod } d_1d_2)} {\sum} A(n, 1) v\Big( \frac{f_1f_2fn}{N}\Big)
\end{split}
\end{displaymath}
We conclude (for $N \geq 10$) that 
\begin{displaymath}
\begin{split}
T_1^{\pm} \ll   & \frac{D}{N^{1/2}}% \sum_{d_2 \asymp D_2}  
\sum_{d_2 f_1 g_1 g_2g_3  \ll Q/D % \ll Q/d_2D 
} |A(g_2, g_1)| \Big| \sum_{d_1}v\Big(\frac{d_1}{D}\Big)W\Big( \frac{d_1d_2f_1g_1g_2g_3}{Q}\Big)\underset{ Bn \equiv \pm 1\, (\text{mod } d_1d_2)}{ \sum} A(n, 1) v\Big( \frac{Bn}{N}\Big)\Big| 
\end{split}
\end{displaymath}
with
$$B = f_1^3 g_1^2 g_2.$$
%At this point we also localize all other variables   $f_1 \asymp F_1$, $g_j \asymp G_j$, and we think of $d_2, f_1, g_1, g_2, g_3$ as small auxiliary variables. We generally have
%$$D D_2 F_1G_1G_2G_3 \ll Q.$$
%We are now in a position to apply Theorem \ref{thm1a} 
We write the congruence as an equation
$$Bn - d_2m = \pm 1, \quad m = r d_1,$$
and we can insert a redundant smooth weight function localizing the new variable  $r \asymp N/Dd_2$. Hence we are in a position to apply  Theorem \ref{thm1a}, getting
\begin{align}\label{T1bound}
	T_1^{\pm}\ll Q^{\varepsilon}  \frac{D}{N^{1/2}}  \frac{Q}{D} N^{41/42}= Q^{1+\varepsilon }N^{10/21} \ll Q^{12/7+\varepsilon} X^{10/21}
\end{align}
by \eqref{Nbound}.

We choose $X=Q^{45/82}$  and combine \eqref{T2bound},  \eqref{mainterm} and \eqref{T1bound} to finish the proof of Theorem~\ref{thm2}. 

\end{document}